\date{\today}
\documentclass[10pt]{amsart}
\usepackage{amssymb,amsfonts,amsthm,amsmath,verbatim,lscape,color,tensor,url,todonotes}
\usepackage{stmaryrd}
\usepackage[margin=1.5in]{geometry}
\usepackage[enableskew]{youngtab}
\usepackage[all]{xy}
\usepackage{enumerate}
\usepackage{mathrsfs}
\usepackage{caption}
\def\biblio{\bibliography{bibliography}\bibliographystyle{alpha}}

\usepackage[pagebackref]{hyperref}

\usepackage{subfiles}
\usepackage{graphicx}

\hypersetup{%
  bookmarksnumbered=true,%
  bookmarks=true,%
  colorlinks=true,%
  linkcolor=blue,%
  citecolor=blue,%
  filecolor=blue,%
  menucolor=blue,%
  pagecolor=blue,%
  urlcolor=blue,%
  pdfnewwindow=true,%
  pdfstartview=FitBH}

\usepackage[nameinlink,capitalise,noabbrev]{cleveref}
\newtheorem{thm}{Theorem}[section]
\newtheorem*{thm*}{Theorem}
\newtheorem{cor}[thm]{Corollary}
\newtheorem{prop}[thm]{Proposition}
\newtheorem{lem}[thm]{Lemma}

\newtheorem{convention}[thm]{Convention}

\theoremstyle{definition}
\newtheorem{defn}[thm]{Definition}

\newtheorem{ex}[thm]{Example}

\theoremstyle{remark}
\newtheorem{rem}[thm]{Remark}
\DeclareMathOperator{\Alg}{Alg}
\DeclareMathOperator{\CAlg}{CAlg}
\DeclareMathOperator{\im}{im}

\DeclareMathOperator{\Hom}{Hom}

\DeclareMathOperator{\colim}{colim}

\DeclareMathOperator{\End}{End}

\DeclareMathOperator{\Id}{Id}
\DeclareMathOperator{\Tot}{Tot}
\DeclareMathOperator{\Ind}{Ind}

\DeclareMathOperator{\fib}{fib}

\DeclareMathOperator{\Mod}{\mathrm{Mod}}

\DeclareMathOperator{\Pic}{Pic}

\DeclareMathOperator{\Sp}{Sp}
\DeclareMathOperator{\Loc}{Loc}

\DeclareMathOperator{\Cat}{Cat_{\infty}}

\DeclareMathOperator{\Catomega}{Cat_{\infty}^{\omega}}
\DeclareMathOperator{\Fun}{Fun}

\DeclareMathOperator{\Thick}{Thick}

\DeclareMathOperator{\alg}{alg}

\DeclareMathOperator{\Ch}{Ch}

\DeclareMathOperator{\tors}{tors}
\DeclareMathOperator{\comp}{comp}

\DeclareMathOperator{\Franke}{Fr}

\newcommand{\cA}{\mathcal{A}}

\newcommand{\Frnp}{\Franke_{n,p}}

\newcommand{\noloc}{\;\mathord{:}\,}
\newcommand{\N}{\mathbb{N}}
\newcommand{\Z}{\mathbb{Z}}
\newcommand{\F}{\mathbb{F}}
\newcommand{\Q}{\mathbb{Q}}

\newcommand{\E}{E_{n,p}}

\newcommand{\bE}{\mathbb{E}}

\newcommand{\cC}{\mathcal{C}}
\newcommand{\cD}{\mathcal{D}}
\newcommand{\cF}{\mathcal{F}}
\newcommand{\cR}{\mathcal{R}}
\newcommand{\cK}{\mathcal{K}}
\newcommand{\cL}{\mathcal{L}}

\newcommand{\cP}{\mathcal{P}}
\newcommand{\cS}{\mathcal{S}}

\newcommand{\lra}[1]{\overset{#1}{\longrightarrow}}

\newcommand{\doubleflat}{\flat\kern-1.0pt\flat}

\newcommand{\Prod}[1]{{\prod}_{#1}}

\newcommand{\Colim}[1]{\underset{#1}{\colim}}

\newcommand{\Spnp}{\Sp_{n,p}}

\makeatletter
\let\c@equation\c@thm
\makeatother

\numberwithin{equation}{section}
\Crefname{figure}{Figure}{Figures}

\newcommand{\id}{\mathrm{id}}

\newcommand{\cI}{\mathcal{I}}
\newcommand{\cJ}{\mathcal{J}}
\newcommand{\CC}{\overline{\mathcal{C}}}
\newcommand{\Cayley}{\mathrm{\mathfrak{C}}}

\DeclareMathOperator{\Cell}{Cell}

\DeclareMathOperator{\PicCell}{PicCell}

\newcommand{\cFrnp}{\widehat{\Franke}_{n,p}}
\newcommand{\cSpnp}{\widehat{\Sp}_{n,p}}

\newcommand{\cMod}{\widehat{\Mod}}
\newcommand{\tMod}{\Mod^{\tors}}
\newcommand{\csmash}{\widehat{\otimes}}
\newcommand{\K}{K_p(n)}

\newcommand{\bP}{\mathcal{P}}
\newcommand{\fm}{\mathfrak{m}}
\newcommand{\cotimes}{\hat{\otimes}}
\newcommand{\ccC}{\widehat{\cC}}
\newcommand{\A}{A_{n,p}}
\newcommand{\algkappa}{\kappa_{\mathrm{alg},p}}
\newcommand{\Malg}{M_{\mathrm{alg}}}
\newcommand{\topkappa}{\kappa_{\mathrm{top},p}}
\newcommand{\Mtop}{M_{\mathrm{top}}}

\newcommand{\cG}{\mathcal{G}}
\newcommand{\cH}{\mathcal{H}}

\DeclareMathOperator{\cell}{cell}

\newcommand{\Mnp}{{\Sp_{n,p}^{\tors}}}
\newcommand{\tFrnp}{\Frnp^{\tors}}
\newcommand{\M}{M_{n,p}}
\newcommand{\cB}{\mathcal{B}}

\date{\today}

\begin{document}
\title{Monochromatic homotopy theory is asymptotically algebraic}

\author{Tobias Barthel}
\thanks{The first author was partially supported by the DNRF92 and the European Unions Horizon 2020 research and innovation programme under the Marie Sklodowska-Curie grant agreement No.~751794.}
\address{University of Copenhagen\\ Copenhagen, Denmark} 
\email{tbarthel@math.ku.dk}

\author{Tomer M. Schlank}
\address{The Hebrew University of Jerusalem \\ Jerusalem, Israel} 
\thanks{The second author is supported by the Alon Fellowship and ISF1588/18.}
\email{tomer.schlank@mail.huji.ac.il}

\author{Nathaniel Stapleton}
\address{University of Kentucky\\ Lexington, USA} 
\thanks{All three authors would like to thank the Isaac Newton Institute for Mathematical Sciences,
Cambridge, for support and hospitality during the programme \emph{Homotopy Harnessing Higher Structures}, where work on this paper was undertaken. This work was supported by EPSRC grant no
EP/K032208/1.}

\email{nat.j.stapleton@gmail.com}

\begin{abstract} 
In previous work, we used an $\infty$-categorical version of ultraproducts to show that, for a fixed height $n$, the symmetric monoidal $\infty$-categories of $\E$-local spectra are asymptotically algebraic in the prime $p$. In this paper, we prove the analogous result for the symmetric monoidal $\infty$-categories of $\K$-local spectra, where $\K$ is Morava $K$-theory at height $n$ and the prime $p$. This requires $\infty$-categorical tools suitable for working with compactly generated symmetric monoidal $\infty$-categories with non-compact unit. The equivalences that we produce here are compatible with the equivalences for the  $\E$-local $\infty$-categories.
\end{abstract}

\date{\today}


\maketitle

{\hypersetup{linkcolor=black}\tableofcontents}

\def\biblio{}

\section{Introduction}

Chromatic homotopy theory describes how the stable homotopy category can be built out of irreducible building blocks depending on a prime $p$ and a height $n$ called the $K(n)$-local categories. These categories have peculiar categorical properties that are not visible in the global context of the stable homotopy category. In particular, the $K(n)$-local category inherits a symmetric monoidal structure from the stable homotopy category whose invariants, such as Picard groups, have been an active area of research. When $n$ is fixed and $p$ increases, the $K(n)$-local category simplifies in various ways. For instance, the Picard groups are purely algebraic \cite{pp2} and certain spectral sequences grow sparser leading to controllable calculations and the capacity to construct basic spectra.

In \cite{ultra1}, we categorify a similar simplification for the $E_n$-local categories $\Sp_{n,p}$. This is accomplished by introducing an algebraic analogue of $\Sp_{n,p}$, called $\Franke_{n,p}$, and by constructing an equivalence of symmetric monoidal $\infty$-categories 
\begin{equation} \label{introequiv}
\Prod{\cF}^{\Pic} \Spnp \simeq \Prod{\cF}^{\Pic} \Frnp,
\end{equation}
where $\cF$ is a non-principal ultrafilter on the set of prime numbers and $\Prod{\cF}^{\Pic}$ denotes the $\Pic$-generated protoproduct of \cite[Section 3]{ultra1}. This equivalence allows one to move certain results in $\Frnp$ to $\Spnp$ for large enough primes as shown in \cite[Section 6]{ultra1}. For the purpose of applications, it is important to have a $K(n)$-local version of the equivalence above and to understand how it relates to the $E_n$-local equivalence. That is the purpose of the current paper, which may be viewed as a sequel to \cite{ultra1}.

In this paper, we build a monochromatic analogue of $\Frnp$, called $\cFrnp$, and extend the $\Pic$-generated protoproduct construction to include the $K(n)$-local categories $\cSpnp$ and the categories $\cFrnp$.
\begin{thm} \label{mainthm1}
There is an equivalence of symmetric monoidal $\infty$-categories
\[
\Prod{\cF}^{\Pic} \cSpnp \simeq \Prod{\cF}^{\Pic} \cFrnp.
\]
In particular, this induces an equivalence of Picard $\infty$-groupoids
\[
\Pic(\Prod{\cF}^{\Pic} \cSpnp) \simeq \Pic(\Prod{\cF}^{\Pic} \cFrnp).
\]
\end{thm}

Moreover, in \cref{thm:expicgen}, we relate these Picard $\infty$-groupoids to the ultraproduct of the Picard groups of the $K(n)$-local categories at a nonprincipal ultrafilter.

The proof of Theorem \ref{mainthm1} is not purely formal primarily because the $K(n)$-local category $\cSpnp$ behaves, in many ways, quite differently than the $E_n$-local category. In particular, it is a naturally occurring example of a symmetric monoidal compactly generated $\infty$-category in which the unit is not compact. This leads to real difficulties that must be surmounted in order to produce a well-behaved $\Pic$-generated protoproduct. 

The construction of the $\Pic$-generated protoproduct in \cite{ultra1} requires that the invertible objects in the input $\infty$-categories are compact. This is not true in $\cSpnp$ or $\cFrnp$. We define a notion of the $\Pic$-generated protoproduct that does not require the invertible objects to be compact. In general, this construction produces non-unital symmetric monoidal $\infty$-categories. To address this issue in our situation, we make use of the fact that the units in $\cSpnp$ and $\cFrnp$ can be built from compact objects uniformly in the prime. In this way we obtain the symmetric mondoial $\infty$-categories in the equivalence of Theorem \ref{mainthm1}. 

The proof of Theorem \ref{mainthm1} follows the same steps as the proof of the main theorem in \cite{ultra1}. The first step was to produce equivalences
\begin{equation} \label{ultra1eq2}
\Prod{\cF}^{\flat} \Mod_{\E^{\otimes s}} \simeq \Prod{\cF}^{\flat} \Mod_{(\E^{\otimes s})_*}
\end{equation}
and the second step was to deduce the main theorem by descent along these equivalences.

From the point of view of local duality contexts~\cite{heardvalenzuela_localduality}, the $K(n)$-local category can be realized as a torsion subcategory in the $E(n)$-local category. Thus, to get the monochromatic analogue of the equivalence \eqref{ultra1eq2}, we restrict it to an equivalence between suitable torsion subcategories. For a more detailed explanation, see the next section.

\subsection*{Acknowledgements} It is a pleasure to thank Rune Hausgeng for helpful conversations.

\section{Main theorem and outline of the proof}

\subsection{The main theorem}

The goal of this paper is to prove the following theorem:

\begin{thm}\label{thm:knmain}
For any non-principal ultrafilter $\cF$ on $\bP$, there is a symmetric monoidal equivalence of compactly generated $\Q$-linear stable $\infty$-categories
\[
\Prod{\cF}^{\Pic} \cSpnp \simeq \Prod{\cF}^{\Pic} \cFrnp.
\]
\end{thm}

The notation in the statement of the theorem requires explanation. The $\infty$-category $\cSpnp$ is the $K(n)$-local category, which can be constructed as the localization of $\Spnp$ by the $\E$-localization of a type $n$ complex. Analogously, the $\infty$-category $\cFrnp$ is the localization of $\Frnp$ at $(\E)_{\star}/I_{n,p}$ (see Section 4.1 of \cite{ultra1} for a discussion of the $\star$-operator and formality) and $\Frnp$ is the underlying $\infty$-category of the category of quasi-periodic complexes of comodules over $(\pi_0 \E, \pi_0(\E \wedge \E))$ periodized with respect to the comodule $\pi_2 \E$. The Pic-generated protoproduct is a generalization of the Pic-generated protoproduct of \cite{ultra1} to symmetric monoidal compactly generated $\infty$-categories in which the unit is not necessarily compact.

In \cite[Section 3]{ultra1}, we introduce the notion of a protoproduct of compactly generated $\infty$-categories. The protoproduct takes in a collection of compactly generated $\infty$-categories equipped with a filtration on the subcategory of compact objects and produces a compactly generated $\infty$-category. The $\Pic$-generated protoproduct of \cite[Section 3]{ultra1} is the special case where the $k$th stage in the filtration consists of compact objects that can be built out of at most $k$ elements in the Picard group of the $\infty$-category. Since the unit is assumed to be compact in \cite{ultra1}, the unit is contained in every filtration degree and the $\Pic$-generated ultraproduct is symmetric monoidal by construction. In this paper, we are concerned with $\cSpnp$ and $\cFrnp$, which are symmetric monoidal compactly generated $\infty$-categories with non-compact unit. However, intuition suggests that the Pic-generated protoproduct of these $\infty$-categories should still be symmetric monoidal: the units in these $\infty$-categories can be built out of compact objects in a prime-independent way. A large part of the work in \cref{ssec:four} goes towards proving that the $\infty$-categories are also unital.

\subsection{Leitfaden for the proof and conventions} \label{leitfaden}

The next diagram summarizes the various steps in the proof of \Cref{thm:knmain}:

\[
\resizebox{\textwidth}{!}{
\xymatrix{
\Prod{\cF}^{\Pic} \cSpnp \ar@{^{(}->}[r] \ar@{-->}[ddd]^{\sim} & \Tot\Prod{\cF}^{\Pic}\cMod_{(\E^{\otimes \bullet+1})}  \ar[r]^-{\sim} &\Tot\Prod{\cF}^{\flat}\cMod_{(\E^{\otimes \bullet+1})} \ar[r]^-{\sim} &  \Tot\Prod{\cF}^{\flat}\tMod_{(\E^{\otimes \bullet+1})} \ar[d]^{\sim} \\
 & \Prod{\cF}^{\Pic}\Spnp \ar[ul]  \ar@{^{(}->}[r] \ar[d]^{\sim} & \Tot\Prod{\cF}^{\flat}\Mod_{(\E^{\otimes \bullet+1})}  \ar[d]^{\sim} \ar[r] & \Tot(\Prod{\cF}^{\flat}\Mod_{(\E^{\otimes \bullet+1})})^{\tors} \ar[d]^{\sim}   \\
 & \Prod{\cF}^{\Pic} \Frnp \ar@{^{(}->}[r] \ar[dl] & \Tot\Prod{\cF}^{\flat}\Mod_{(\E^{\otimes \bullet+1})_{\star}} \ar[r]  & \Tot(\Prod{\cF}^{\flat} \Mod_{(\E^{\otimes \bullet+1})_{\star}})^{\tors}  \\
\Prod{\cF}^{\Pic} \cFrnp \ar@{^{(}->}[r] & \Tot\Prod{\cF}^{\Pic}\cMod_{(\A^{\otimes \bullet+1})} \ar[r]^-{\sim} & \Tot\Prod{\cF}^{\flat} \cMod_{(\A^{\otimes \bullet+1})} \ar[r]^-{\sim} & \Tot\Prod{\cF}^{\flat} \tMod_{(\E^{\otimes \bullet+1})_{\star}}  \ar[u]_{\sim}}}
\]

In this diagram:
\begin{itemize}
	\item $\E$ is Morava $E$-theory at height $n$ and the prime $p$ and $\K$ denotes the corresponding Morava $K$-theory spectrum. Implicitly in this notation is the choice of the Honda formal group law over $\F_{p^n}$. By the Goerss--Hopkins--Miller theorem, $\E$ has a canonical structure as a $\K$-local $\mathbb{E}_{\infty}$-ring spectrum.
	\item $\A = \underline{P}((\E)_0\E)$ is the commutative algebra object in the symmetric monoidal $\infty$-category $\Frnp$ studied in \cite[Section 5.3]{ultra1}.
	\item $\cMod_{(\E^{\otimes s})}$ denotes the symmetric monoidal $\infty$-category of modules over $L_{\K}(\E^{\otimes s})$ in $\cSpnp$, see \Cref{ssec:ex} for a more precise definition.
	\item $\cMod_{(\A^{\otimes s})}$ denotes the symmetric monoidal $\infty$-category of modules over the completion of $\A^{\otimes s}$ as an object in $\cFrnp$. 
	\item Torsion objects in this context refers to torsion objects in the sense of local duality contexts~\cite{heardvalenzuela_localduality}, and the corresponding categories of torsion objects are indicated by a superscript ``$\tors$''. 
	\item The protoproduct of the form $\Prod{\cF}^{\Pic}$ is a generalization of the Pic-protoproduct of \cite[Section 3.5]{ultra1} to the $\infty$-categories of interest in this paper. Its construction and properties are given in more detail in \Cref{ssec:four}.
	\item Similarly to the Pic-protoproduct, the protoproduct of the form $\Prod{\cF}^{\flat}$ is a generalization of the cell-protoproduct of \cite[Section 3.5]{ultra1} to the $\infty$-categories of interest in this paper.
\end{itemize}

The equivalences in the top (topological) and bottom (algebraic) part of the diagram are established in parallel, so we will only comment on the former:

\begin{itemize}
	\item The symmetric monoidal equivalence 
	\[
	\Tot\Prod{\cF}^{\Pic}\cMod_{(\E^{\otimes \bullet+1})} \lra{\sim} \Tot\Prod{\cF}^{\flat}\cMod_{(\E^{\otimes \bullet+1})}
	\] 
	uses the Picard group computation of $\cMod_{\E}$ together with a cosimplicial detection result proved in \cite{ultra1}, see \cref{cor:totalization}. 
	\item The symmetric monoidal equivalence 
	\[
	\Tot\Prod{\cF}^{\flat}\cMod_{(\E^{\otimes \bullet+1})} \lra{\sim}  \Tot\Prod{\cF}^{\flat}\tMod_{(\E^{\otimes \bullet+1})}
	\] 
	follows from applying the protoproduct construction to the local duality equivalence of \cref{prop:etheory}, following~\cite{heardvalenzuela_localduality}.
	\item The symmetric monoidal equivalence 
	\[
	\Tot\Prod{\cF}^{\flat}\tMod_{(\E^{\otimes \bullet+1})} \lra{\sim} \Tot(\Prod{\cF}^{\flat}\Mod_{(\E^{\otimes \bullet+1})})^{\tors}
	\]
	follows from our study of torsion objects in protoproducts in \cref{ssec:prototorsion}.
\end{itemize}

In order to finish the proof, we need to relate the topological and algebraic sides of the diagram. This is achieved in two steps:
\begin{itemize}
\item By \cref{sec:end}, there is a symmetric monoidal equivalence 
\[
\Tot(\Prod{\cF}^{\flat}\Mod_{(\E^{\otimes \bullet+1})})^{\tors} \lra{\sim} \Tot(\Prod{\cF}^{\flat} \Mod_{(\E^{\otimes \bullet+1})_{\star}})^{\tors}.
\] 
\item By making use of the uniform descent results described in \cite[Sections 5.1 and 5.2]{ultra1}, \cref{cor:totalization} provides equivalences 
\[
\Prod{\cF}^{\Pic} \cSpnp \simeq \Loc \Pic \Tot\Prod{\cF}^{\Pic}\cMod_{(\E^{\otimes \bullet+1})}
\]
and 
\[
\Prod{\cF}^{\Pic} \cFrnp \simeq \Loc \Pic \Tot\Prod{\cF}^{\Pic}\cMod_{(\A^{\otimes \bullet+1})}.
\]
Since each of these equivalences is symmetric monoidal and colimit preserving, there is an induced equivalence
\[
\Prod{\cF}^{\Pic} \cSpnp \lra{\sim} \Prod{\cF}^{\Pic} \cFrnp.
\] 
\end{itemize}
The commutativity of the diagram is established in \cref{ssec:comparison}. In particular, this implies that the equivalence produced here is compatible with the equivalence in the main theorem of \cite{ultra1}. Finally, we describe some further conventions used throughout the paper:  

\begin{itemize}
	\item We write $\bP$ for the set of prime numbers.
	\item We write $\Hom$ for mapping spectra in stable $\infty$-categories.
	\item The $\infty$-category of commutative monoids in a symmetric monoidal $\infty$-category $\cC$ will be denoted by $\CAlg(\cC)$ and we refer to its objects as commutative algebras in $\cC$. For $\cC = \Sp$ equipped with its natural symmetric monoidal structure, we usually say $\mathbb{E}_{\infty}$-ring spectrum or $\mathbb{E}_{\infty}$-ring instead of commutative algebra.
	\item Let $\Pr^L$ be the $\infty$-category of presentable $\infty$-categories and left adjoint functors, and let $\Catomega$ denote the $\infty$-category of compactly generated $\infty$-categories and colimit preserving maps that preserve compact objects. 
	\item Given two $\infty$-categories, $\cC$ and $\cD$, let $\Fun^L(\cC,\cD)$ be the $\infty$-category of colimit preserving functors from $\cC$ to $\cD$.
	\item A presentable symmetric monoidal $\infty$-category $\cC = (\cC,\otimes)$ is called presentably symmetric monoidal if $\cC \in \CAlg(\Pr^L)$.
	\item A compactly generated symmetric monoidal $\infty$-category is an object in $\CAlg(\Catomega)$ and a symmetric oidal $\infty$-category is an object in $\Alg_{\bE_{\infty}^{\text{nu}}}(\Catomega)$, where $\bE_{\infty}^{\text{nu}}$ is the non-unital version of the $\bE_{\infty}$-operad (See \cite[Section 5.4.4]{ha}).
	\item If $\cC$ is a presentably symmetric monoidal stable $\infty$-category and $A$ is a commutative algebra in $\cC$, then $\Mod_A(\cC)$ denotes the stable $\infty$-category of modules over $A$ in $\cC$. In the case $\cC = \Sp$, we will write $\Mod_A$ instead of $\Mod_A(\cC)$ for simplicity. Similarly, we write $\CAlg_A(\cC)$ for the $\infty$-category of commutative $A$-algebras in $\cC$ and omit the $\infty$-category $\cC$ when it is clear from context and in particular whenever $\cC = \Sp$. 
 	\item A module $M$ over a commutative ring $R$ is said to be $I$-torsion for an ideal $I \subseteq R$ if any element $m\in M$ is annihilated by a power of $I$.
\end{itemize}

\section{Torsion and completion}\label{sec:localduality}

In this section, we collect some material on categories of torsion and complete objects; in order to develop the theory for the algebraic and topological sides in parallel, we formulate our results in general terms. 

\subsection{The context} \label{Section3.1}

Throughout this section, suppose $\cC=(\cC,\otimes,\mathbf{1}_{\cC})$ is a presentably symmetric monoidal stable $\infty$-category which is compactly generated by its invertible objects. In particular, this implies that the dualizable objects in $\cC$ can be identifies with the compact objects in $\cC$. Let $F \in \cC^{\omega}$ be a compact object. We will sometimes refer to the pair $(\cC,F)$ as a local duality context.

Given a pair $(\cC,F)$ as above, consider the localizing ideal $\Loc^{\otimes}(F)$ in $\cC$ generated by $F$, which coincides with the localizing subcategory of $\cC$ generated by $F \otimes \Pic(\cC)$. The canonical inclusion of $\Loc^{\otimes}(F)$ into $\cC$ admits a right adjoint $\Gamma_F$ by the adjoint functor theorem. The symmetric monoidal  product on $\cC$ restricts to a symmetric monoidal product on $\Loc^{\otimes}(F)$ whose unit is given by $\Gamma_F \mathbf{1}$. Furthermore, write $L_X\colon \cC \to \cC$ for the Bousfield localization functor associated to $X$ given by inverting the $(X \otimes -)$-equivalences. We define the $\infty$-category $\ccC_X = L_X\cC$ of $X$-complete objects in $\cC$ as the essential image of $L_X$; if no confusion is likely to arise, we will also omit the subscript $X$ from the notation and write $\ccC$ for $\widehat{\cC}_X$. 

\begin{ex}\label{ex:enlocal}
Given a nonnegative integer $n$ and a prime $p$, an example of a local duality context is given by the $\E$-local category $\Spnp$ with $F = L_{\E}F(n,p)$ for some finite type $n$ spectrum $F(n,p)$. In this case, $L_F$ is equivalent to $\K$-localization and $\widehat{\cC}=\cSpnp$ is the category of $\K$-local spectra. 
\end{ex}

The next result summarizes the key features of local duality contexts that we will use throughout this paper; for the proofs, see \cite{axiomatic} and \cite{heardvalenzuela_localduality}.

\begin{prop}\label{prop:localdualitycontext}
The following holds for a local duality context $(\cC,F)$ as defined above:
	\begin{enumerate}
		\item The category $\ccC$ is a presentable stable $\infty$-category compactly generated by $F \otimes \Pic(\cC)$. Furthermore, the canonical inclusion functor $\ccC \to \cC$ is fully faithful and preserves limits and compact objects, while colimits in $\ccC$ are computed by applying $L_F$ to the corresponding colimit in $\cC$. 
		\item The localized monoidal structure $-\csmash- = L_{F}(-\otimes-)$ equips $\ccC$ with the structure of a presentably symmetric monoidal $\infty$-category with unit $L_{F}\mathbf{1}_{\cC}$. Moreover, if $X \in \Thick_{\cC}^{\otimes}(F)$, the thick tensor ideal, then $L_FX \simeq X$, so $X \cotimes Y \simeq X \otimes Y$ for all $Y \in \cC$. 
		\item The $\infty$-groupoid $L_F\Pic(\cC) \subseteq \Pic(\ccC)$ generates $\ccC$. In particular, $\ccC = \Loc\Pic(\ccC)$.
		\item The compact objects in $\ccC$ consist of the thick subcategory $\Thick(F \otimes \Pic(\cC)) = \Thick_{\cC}^{\otimes}(F)$. 
		\item There exists a cofiltered system of compact objects $(M_i(F))_{i \in I}$ in $\ccC$ and a natural equivalence $L_FX \simeq \lim_{I} M_i(F) \otimes X$ for all $X \in \ccC$. 		
		\item The localization functor $L_F$ induces a symmetric monoidal equivalence $\Loc^{\otimes}(F) \simeq \ccC$. In particular, there exists a nonunital symmetric monoidal colimit preserving functor $\ccC \to \cC$. 
	\end{enumerate}
\end{prop}
\begin{proof}
The first two claims follow directly from the construction of $\ccC$ as a localization, see for example~\cite[Theorem 2.21(3)]{heardvalenzuela_localduality}, while the identification of a set of compact generators is a consequence of (6). Since $\Pic(\cC)$ generates $\cC$, Part (3) follows from adjunction. Part of Part (4) is \cite[Lemma (2.15)]{heardvalenzuela_localduality}. To show that $\Thick(F \otimes \Pic(\cC)) = \Thick_{\cC}^{\otimes}(F)$, it is enough to show that $\Thick(F \otimes \Pic(\cC))$ is a thick tensor ideal. This follows from the fact that $\cC^{\omega} = \Thick(\Pic(\cC))$. Part (5) is a consequence of \cite[Equation (2.30)]{heardvalenzuela_localduality} and the local duality equivalence~\cite[Theorem 2.21(4)]{heardvalenzuela_localduality}.

The first part of the final claim is the content of \cite[Theorem 2.21(3) and Proposition 2.34]{heardvalenzuela_localduality}. The desired functor $\ccC \to \cC$ is the composite of the equivalence with the canonical inclusion $\Loc^{\otimes}(F) \to \cC$. Note that the latter functor is left adjoint to a symmetric monoidal functor and hence has the structure of a symmetric colax monoidal functor. In order to see that it is in fact nonunital strict symmetric monoidal, it thus suffices to pass to homotopy categories, where it can be checked directly. 
\end{proof}

\begin{rem}
Note that the unit object $L_{F}\mathbf{1}_{\cC} \in \ccC$ is in general not compact, as \Cref{ex:enlocal} for $n>0$ demonstrates. 
\end{rem}

\begin{cor}\label{cor:savetheday}
Let $(\cC,F)$ be a local duality context, then $\ccC^{\omega} \subseteq \Thick(\Pic(\ccC))$. 
\end{cor}
\begin{proof}
Since $F \in \cC$ is compact, \Cref{prop:localdualitycontext}(1) gives $F \in \Thick(\Pic(\cC))$. It thus follows from \Cref{prop:localdualitycontext}(3) that
\[
F \otimes \Pic(\cC) \subseteq \Thick(L_F\Pic(\cC)) \subseteq \Thick(\Pic(\ccC)).
\]
By \Cref{prop:localdualitycontext}(4), $\ccC^{\omega}  = \Thick(F \otimes \Pic(\cC))$, which implies the claim.

\end{proof}

\subsection{Modules and local duality}\label{ssec:localduality}

Let $(\cC,F)$ be a local duality context and let $A \in \CAlg(\cC)$ be a commutative algebra in $\cC$. Define $\Mod_A = \Mod_A(\cC)$ to be the $\infty$-category of $A$-modules internal to $\cC$ and $\cMod_{A} = \Mod_{L_FA}(\ccC)$ to be the $\infty$-category of $L_FA$-modules internal to $\ccC = \widehat{\cC}_F$. A standard argument (see e.g., \cite[Lemma 2.25]{arthur}  applied to $A \otimes -\colon \cC \to \Mod_A(\cC)$) shows that the $\infty$-category $\Mod_A$ is compactly generated by $A \otimes \Pic(\cC)$. The $\infty$-category $\cMod_{A}$ is compactly generated by $L_FA \cotimes (F \otimes \Pic(\cC))$ by \Cref{prop:localdualitycontext}(1). An object $X \in \cC$ is called full if the functor $X \otimes - \colon \cC \to \cC$ is conservative. The next proposition generalizes a result due to Hovey \cite[Corollary 2.2]{hoveycsc}. 

\begin{prop}\label{prop:comparelocfun}
Let $\cC$ be a monogenic compactly generated presentably symmetric monoidal stable $\infty$-category, $A \in \CAlg(\cC)$ a commutative algebra, and $F \in \cC^{\omega}$ a compact object in $\cC$, then there is a natural equivalence
\[
\xymatrix{L_{A \otimes F} \ar[r]^-{\simeq} & L_FL_A}
\] 
of endofunctors on $\cC$. In particular, if $A$ is full, then $L_{A\otimes F} \simeq L_F$. 
\end{prop}
\begin{proof}
There is a natural transformation $\mathrm{id} \to L_FL_A$ obtained by combining the unit maps of $L_F$ and $L_A$. We need to show that, for any $X \in \cC$, the canonical morphism $X \to L_FL_AX$ is an $(F \otimes A)$-equivalence to an $(F\otimes A)$-local object. The map in question factors as $X \to L_AX \to L_F(L_AX)$, i.e., an $A$-equivalence followed by an $F$-equivalence, hence the composite is an $(F \otimes A)$-equivalence.

In order to show that $L_FL_AX$ is $(F\otimes A)$-local, consider an $(F\otimes A)$-acyclic object $Z$. By Parts (4) and (5) of \cref{prop:localdualitycontext}, $L_F(-) \simeq \lim_{i\in I}(M_i(F)\otimes -)$ for a filtered diagram consisting of objects $M_i(F) \in \Thick_{\cC}^{\otimes}(F) \subset \ccC$ for all $i \in I$. Consider the objects $M_i(F)$ as objects in $\cC$ and let $D(-)$ represent the monoidal dual in $\cC$. Since $(Z\otimes K) \otimes A \simeq Z \otimes (K \otimes A) \simeq 0$ for any $K \in \Thick_{\cC}^{\otimes}(F)$ and since $\Thick_{\cC}^{\otimes}(F)$ is closed under taking monoidal duals, our assumption on $Z$ implies that the object $Z \otimes DM_i(F)$ is $A$-acyclic for any $i \in I$. Thus we have equivalences
\[
\Hom(Z,M_i(F) \otimes L_AX) \simeq \Hom(Z \otimes DM_i(F), L_AX) \simeq 0.
\]
This implies that $\Hom(Z, L_FL_AX) \simeq 0$, hence $L_FL_AX$ is $(F\otimes A)$-local as claimed.
\end{proof}

The $\infty$-category of $F$-torsion $A$-modules $\tMod_A$ is given as the localizing ideal in $\Mod_A(\cC)$ generated by $A \otimes F$, while the $\infty$-category $\Mod_{A}^{\comp}$ of $F$-complete $A$-modules in $\cC$ is by definition the essential image of the Bousfield localization $L_{A\otimes F}^{A}\colon \Mod_A \to \Mod_A$ with respect to $A \otimes F$ constructed in $\Mod_A$. The next result lifts local duality with respect to $F$ in $\cC$ to the corresponding module categories over $A$.

\begin{prop}\label{prop:ctelocalduality}
The localization functors induce symmetric monoidal equivalences
\[
\xymatrix{\tMod_{A} \ar[r]_-{\simeq}^-{L_{A\otimes F}^A} & \Mod_{A}^{\comp} \ar[r]_-{\simeq}^-{L_F} & \cMod_A.}
\]
\end{prop}
\begin{proof}
The first symmetric monoidal equivalence is an instance of local duality, see \Cref{prop:localdualitycontext}(6). Bousfield localization at $F$ induces a symmetric monoidal functor between module categories $L_F\colon \Mod_A(\cC) \to \Mod_{L_FA}(L_F\cC)$. We claim that this functor annihilates $L_{A\otimes F}^A$-acyclic $A$-modules. Indeed, let $M \in \Mod_A(\cC)$ be $A\otimes F$-acyclic, then $0 \simeq M \otimes_A(A\otimes F) \simeq M \otimes F$. We thus obtain a factorization
\[
\xymatrix{\Mod_A(\cC) \ar[r]^-{L_F} \ar[d]_-{L_{A\otimes F}^A} & \Mod_{L_FA}(L_F\cC) \\
L_{A\otimes F}^A\Mod_A(\cC). \ar@{-->}[ru]_{\widetilde{L_F}}}
\] 
The dashed functor $\widetilde{L_F}$ is symmetric monoidal by \cite[Proposition 3.2.2]{Hinich}, and it remains to show that $\widetilde{L_F}$ is an equivalence. 

Note that $L_{A\otimes F}^A\Mod_A(\cC)$ is compactly generated by the $\infty$-groupoid $(A \otimes F) \otimes_A (A \otimes \Pic(\cC) \simeq A \otimes F \otimes \Pic(\cC)$, while $\Mod_{L_FA}(L_F\cC)$ is compactly generated by $L_FA \cotimes (F \otimes \Pic(\cC))$. Since $\widetilde{L_F}$ induces an essentially surjective functor between these $\infty$-groupoids, it suffices to show that $\widetilde{L_F}$ is fully faithful. To this end, first consider for fixed $N \in L_{A\otimes F}^A\Mod_A(\cC)$ the full subcategory $\cS(N)$ spanned by all $M \in L_{A\otimes F}^A\Mod_A(\cC)$ such that the map 
\[
\xymatrix{\Hom_{L_{A\otimes F}^A\Mod_A(\cC)}(M,N) \ar[r] & \Hom_{\Mod_{L_FA}(L_F\cC)}(\widetilde{L_F}M,\widetilde{L_F}N)}
\]
induced by $\widetilde{L_F}$ is an equivalence. Because $\widetilde{L_F}$ preserves all colimits, $\cS(N)$ is closed under colimits, so if $F \otimes A \otimes \Pic(\cC)\subset \cS(N)$, then $L_{A\otimes F}^A\Mod_A(\cC) = \cS(N)$. Because the objects in $F \otimes A \otimes \Pic(\cC)$ are compact, both mapping spectra preserve colimits in the second variable when $M$ is in $F \otimes A \otimes \Pic(\cC)$. Thus we can reduce to proving the claim for $N \in F \otimes A \otimes \Pic(\cC)$. In other words, it suffices to show that the restriction of $\widetilde{L_F}$ to $F \otimes A \otimes \Pic(\cC)$ is fully faithful.

This is now a consequence of the fact that, for any $P \in \Pic(\cC)$, $F \otimes A \otimes P \simeq L_FA \cotimes (F \otimes P)$ in $\cC$ by \Cref{prop:localdualitycontext}(2): for any $P,P' \in \Pic(\cC)$, we have natural equivalences
\begin{align*}
\Hom_{L_{A\otimes F}^A\Mod_A(\cC)}(A\otimes F \otimes P,A\otimes F \otimes P') & \simeq \Hom_{\cC}(F \otimes P,A\otimes F \otimes P') \\
& \simeq \Hom_{\cC}(F \otimes P,L_FA\cotimes F \otimes P')\\
& \simeq \Hom_{\Mod_{L_FA}(L_F\cC)}(L_FA\cotimes F \otimes P,L_FA\cotimes F \otimes P').
\end{align*}
It follows that $\widetilde{L_F}\colon L_{A\otimes F}^A\Mod_A(\cC) \to \Mod_{L_FA}(L_F\cC)$ is a symmetric monoidal equivalence, as claimed.
\end{proof}

\subsection{Examples}\label{ssec:ex}

In this subsection, we exhibit some special features of the two examples we will study in the present paper and use them to deduce a result about Picard groups of module categories.

Suppose $(\cC,F)$ is a local duality context with $\cC$ the full subcategory of the $\infty$-category $\Sp$ of spectra consisting of local objects with respect to a $\otimes$-localization, in particular, the inclusion $C \hookrightarrow \Sp$ is lax monoidal. Let $A \in \CAlg(\cC)$ and assume additionally that the triple $(\cC,F,A)$ satisfies the following conditions:
\begin{enumerate}
	\item $A \in \cC$ is full, i.e., $A \otimes -\colon \cC \to \cC$ is conservative.
	\item $A$ is $F$-complete, i.e., the canonical map $A \to L_FA$ is an equivalence. 
	\item $A$ considered as an $\mathbb{E}_{\infty}$-ring spectrum is even periodic and $\pi_0A$ is a Noetherian complete regular local ring. Furthermore, there exists an ideal $J \subseteq \pi_0A$ which is generated by a regular sequence $(x_1,\ldots,x_n)$ and such that $\kappa(A) := A\otimes F \simeq A/J$ in $\cC$.
\end{enumerate}

Under these assumptions, we get the following result:

\begin{lem}\label{lem:completepicgp}
Bousfield localization at $F$  induces a canonical isomorphism
\[
\xymatrix{\pi_0\Pic(\Mod_{A}) \ar[r]^-{\simeq} & \pi_0\Pic(\cMod_{A}),}
\]
of abelian groups. 
\end{lem}
\begin{proof}
Since Bousfield localization is symmetric monoidal, $L_F$ restricts to a homomorphism
\[
\xymatrix{\phi\colon\pi_0\Pic(\Mod_{A}) \ar[r] & \pi_0\Pic(\cMod_{A})}
\]
of abelian groups and it remains to show that $\phi$ is an isomorphism. The unit $A \in \Mod_{A}$  compact as $\mathbf{1}_{\cC} \in \cC^{\omega}$, so the dualizable and compact objects agree in $\Mod_{A}$; in particular, $\Pic(\Mod_{A}) \subset \Mod_{A}^{\omega}$. Moreover, $A$ is $F$-complete by assumption, so any $P \in \pi_0\Pic(\Mod_{A})$ is already $F$-local and hence $\phi$ is injective. Conversely, by Property (3) above and \cite[Proposition 10.11]{mathew_galois}, any $P \in \pi_0\Pic(\cMod_{A})$ is in $\Thick(A)$, hence we have equivalences 
\[
A \simeq L_FA \simeq P \cotimes_A P^{-1} \simeq P \otimes_A P^{-1}
\] 
in $\Mod_A$, which implies the surjectivity of $\phi$. 
\end{proof}

We now turn to the two main examples of this paper, starting with the topological one. Fix a prime $p$ and an integer $n\ge 0$, and let $\K$ be Morava $K$-theory of height $n$ and prime $p$ and write $\cSpnp = \Sp_{\K}$ for the stable $\infty$-category of $\K$-local spectra. Recall that a $p$-local compact spectrum $F$ is said to be of strict type $n$ if $\K \otimes F \ne 0$ and $K_p(n-1) \otimes F = 0$. For any $n$ and $p$, compact spectra of strict type $n$ exist at $p$: By the periodicity theorem of Hopkins and Smith, there exists a generalized Moore spectrum $M^I(n) = M_p^I(n)$ of type $n$ with Brown--Peterson homology
\[
BP_*M^I(n) \cong BP_*/(p^{i_0},v_1^{i_1},\ldots,v_{n-1}^{i_{n-1}})
\]
for an appropriate sequence $I=(i_0,i_1,\ldots,i_{n-1})$ of positive integers. In particular, strict type $n$ spectra $F$ with $2^n$ cells exist. 

\begin{rem}
In fact, for $p$ large enough with respect to $n$, $\E$-local Smith--Toda complexes exist by \cite[Theorem 6.10]{ultra1}, so may choose $L_nF(n)$ to have $BP$-homology $v_{n}^{-1}BP_*/I_n$. 
\end{rem}

Let $\E$ be Morava $E$-theory of height $n$ at the prime $p$, with coefficients $\pi_0\E = \mathbb{WF}_{p^n}\llbracket u_1,\ldots,u_{n-1}\rrbracket$ and associated category $\Spnp$ of $\E$-local spectra. The pair $(\cC,F) = (\Spnp, L_{\E}F(n))$ form a local duality context. The monochromatic category $\Mnp$ is defined as the localizing ideal of $\Spnp$ generated by $L_{\E}F(n)$; note that by the thick subcategory theorem, the definition of $\Mnp$ does not depends on the choice of $F(n)$. Local duality in the form of  \Cref{prop:localdualitycontext}(6) establishes a canonical symmetric monoidal equivalence 
\begin{equation}\label{eq:toplocalduality}
\xymatrix{\M\colon \cSpnp \ar@<0.5ex>[r] & \Mnp \noloc L_{\K} \ar@<0.5ex>[l],}
\end{equation}
where $\M$ denotes the monochromatic layer functor. 

Let $A=\E \in \CAlg(\Spnp)$. There is a natural equivalence $L_F \simeq L_{\K}$ by \Cref{prop:comparelocfun}, and the formula \Cref{prop:localdualitycontext}(4) takes the concrete form
\begin{equation}\label{eq:completionformula}
L_F(X) \simeq L_{\K}(X) \simeq \lim_IDM^I(n) \otimes X,
\end{equation}
where the limit is indexed on a cofinal sequence of integers as above and $X \in \Spnp$.

Conditions (1)--(3) above hold for $(\Spnp, L_{\E}F(n), \E)$. More generally, the terms appearing in the Amitsur complex of $L_{\E}S^0 \to \E$ satisfy a weak form of these conditions: 

\begin{prop}\label{prop:etheory}
Let $s \ge 1$, then $\K$-localization induces a canonical symmetric monoidal equivalence
\[
\xymatrix{\tMod_{\E^{\otimes s}} \ar[r]^-{\simeq} & \cMod_{\E^{\otimes s}}.}
\]
Furthermore, the $\bE_{\infty}$-ring spectrum $L_{\K}(\E^{\otimes s})$ is even periodic and $\pi_0L_{\K}(\E^{\otimes s})$ is complete with respect to the regular sequence $(p,u_1,\ldots,u_{n-1})$. If $s=1$, then $L_{\K}$ induces an isomorphism
\[
\xymatrix{\Z/2 \cong \pi_0\Pic(\Mod_{\E}) \ar[r]^-{\simeq} & \pi_0\Pic(\cMod_{\E}),}
\]
sending the generator of $\Z/2$ to $\Sigma \E$. 
\end{prop}
\begin{proof}
The first claim is a direct consequence of \Cref{prop:ctelocalduality}. Next, we note that the statement about $L_{\K}(\E^{\otimes s})$ holds for $s=1$. Write $\boxtimes$ for the underived tensor product  and $-\widehat{\boxtimes}-=(-\boxtimes-)_{\fm}^{\wedge}$ for the $\fm$-completed tensor product of $\pi_*E$-modules. By \cite[Corollary 1.24]{bh_ass} and for any $s\ge 1$ there is an isomorphism 
\[
\pi_*L_{\K}(\E^{\otimes s}) \cong (\pi_*L_{\K}(\E^{\otimes 2}))^{\boxtimes_{\pi_*\E} (s-1),}
\]
reducing the proof to the case $s=2$. Recall that $\pi_*L_{\K}(\E^{\otimes 2}) \cong E_*^{\vee}E$ is isomorphic to the algebra of $\pi_*\E$-valued continuous functions on the Morava stabilizer group, see \cite{hoveyopetheory}. In particular, $L_{\K}(\E^{\otimes 2})$ is even periodic, $\pi_0$ is complete, and $(p,u_1,\ldots,u_{n-1})$ is a regular sequence.

The final claim follows from \Cref{lem:completepicgp} and \cite{bakerrichter_inv}. 
\end{proof}

On the algebraic side, the role of the category of $\E$-local spectra is played by Franke's category $\cC = \Frnp$, a symmetric monoidal $\infty$-categorical version of which has been constructed and studied in our earlier paper~\cite{ultra1}. We refer to Sections 4.3 and 5.4 of \emph{loc.~cit.} for the main properties of this category. Note that, in particular, $\Frnp$ is monogenic. The analogue of $L_{\E}F(n)$ is given by the periodization $F=\underline{P}(E_0/I_n)$ of the quotient $E_0/I_n = E_0/(p,\ldots,u_{n-1})$, while we take $\A \in \CAlg(\Frnp)$ to be $\underline{P}(E_0E)$. Here $\underline{P}$ is the quasi-periodization functor of \cite[Corollary 5.31]{ultra1}. The pair $(\Frnp, \underline{P}(E_0/I_n))$ forms a local duality context, because $E_0/I_n$ is a finitely presented $E_0E$-comodule. We write $\cFrnp$ for the corresponding completed version of Franke's category and $\tFrnp = \Loc^{\otimes}(\underline{P}(E_0/I_n))$ for the algebraic analogue of the monochromatic category $\Mnp$. As in the topological context, local duality \Cref{prop:localdualitycontext}(6) produces a canonical pair of mutually inverse symmetric monoidal equivalences
\begin{equation}\label{eq:alglocalduality}
\xymatrix{\tFrnp \ar@<0.5ex>[r] & \cFrnp. \ar@<0.5ex>[l]}
\end{equation}
As in \cite{ultra1}, we can employ Morita theory to compare the algebraic and topological sides. Recall that, for any spectrum $M \in \Sp$, we write $M_{\star} = H\pi_*M$; this construction can be extended to a lax symmetric monoidal functor $(-)_{\star}\colon \Sp \to \Mod_{H\Z}$. With the above notation, the argument of \cite[Lemma 5.34]{ultra1} provides a symmetric monoidal equivalence 
\begin{equation}\label{eq:frmorita}
\xymatrix{\Psi_s\colon \Mod_{\A^{\otimes s}}(\Frnp) \ar[r]^-{\simeq} & \Mod_{(\E^{\otimes s})_{\star}}(\Mod_{H\Z})\simeq \Mod_{(\E^{\otimes s})_{\star}}(\Sp)}
\end{equation}
for any $s\ge 1$. We now investigate how the functor $\Psi_s$ interacts with torsion and complete objects. To this end, let $\cMod_{\A^{\otimes s}} = \Mod_{L_F(\A^{\otimes s})}(L_F\Frnp)$ and $\tMod_{\A^{\otimes s}}(\Frnp)$ be the localizing ideal in $\Mod_{\A^{\otimes s}}(\Frnp)$ generated by $\A^{\otimes s} \otimes F$, as in \Cref{ssec:localduality}. Furthermore, $\tMod_{(\E^{\otimes s})_{\star}}$ is defined to be the localizing ideal in $\Mod_{(\E^{\otimes s})_{\star}}$ generated by $(\E^{\otimes s})_{\star}/(I_n)_{\star}$, i.e., by the $(I_n)_{\star}$-torsion objects. Here, we use the notation $(I_n)_{\star}$ to indicate the algebraic analogue of the ideal $I_n$, so that $(\E)_{\star}/(I_n)_{\star} \cong (\E)_{\star} \otimes E_0/I_n$.

\begin{prop}\label{prop:frtorsionmorita}
For any $s\ge 1$, there is an equivalence $\cMod_{\A^{\otimes s}} \simeq \tMod_{(\E^{\otimes s})_{\star}}$ of symmetric monoidal $\infty$-categories.
\end{prop}
\begin{proof}
The desired equivalence is a composite of the following two symmetric monoidal equivalences:
\[
\xymatrix{\cMod_{\A^{\otimes s}} \ar[r]^-{\simeq} & \tMod_{\A^{\otimes s}}(\Frnp) \ar[r]^-{\simeq} & \tMod_{(\E^{\otimes s})_{\star}}.}
\]
Indeed, applying \Cref{prop:ctelocalduality} to the triple $(\Frnp,P(E_0/I_n),\A^{\otimes s})$ yields the first symmetric monoidal equivalence, while the second one results from \eqref{eq:frmorita} together with the equivalence $\Psi_s(\A^{\otimes s} \otimes F) \simeq (\E^{\otimes s})_{\star}/I_n$. This last equivalence is obtained from the equivalences of $H\Z$-modules:
\begin{align*}
\Psi_s(\A^{\otimes s} \otimes F) & \simeq \Hom_{\Mod_{\A^{\otimes s}}(\Frnp)}(\A^{\otimes s}, \A^{\otimes s} \otimes F) \\
& \simeq \Hom_{\Frnp}(P(E_0), \A^{\otimes s} \otimes F) \\
& \simeq \Hom_{\Ch\Mod_{E_0}}(E_0, UP(E_0) \otimes (E_0E)^{\otimes s-1} \otimes E_0/I_n) \\
& \simeq UP(E_0) \otimes (E_0E)^{\otimes s-1} \otimes E_0/I_n \\
& \simeq (\E)_{\star} \otimes (\E^{\otimes s})_0 \otimes E_0/I_n \\
& \simeq (\E^{\otimes s})_{\star}/(I_n)_{\star},
\end{align*}
which is proven as in \cite[Lemma 5.34]{ultra1}.
\end{proof}

We now establish the analogue of the Picard group computation of \Cref{prop:etheory}. 

\begin{cor}\label{cor:frmodpicard}
Bousfield localization at $F=P(E_0)$ induces an isomorphism
\[
\xymatrix{\Z/2 \cong \pi_0\Pic(\Mod_{\A}(\Frnp)) \ar[r]^-{\simeq} & \pi_0\Pic(\cMod_{\A}),}
\]
sending the generator of $\Z/2$ to $\Sigma \A$. 
\end{cor}
\begin{proof}
By the proof of \Cref{prop:frtorsionmorita}, $\Psi_1$ maps $\A \otimes F$ to $(\E)_{\star}/(I_n)_{\star}$, so \Cref{prop:ctelocalduality} and \eqref{eq:frmorita} provide symmetric monoidal equivalences
\[
\xymatrix{\cMod_{\A} \ar[r]^-{\simeq} & L_{\A \otimes F}\Mod_{\A}(\Frnp) \ar[r]^-{\simeq} & L_{(\E)_{\star}/(I_n)_{\star}}\Mod_{(\E)_{\star}}.}
\]
The isomorphism 
\[
\pi_0\Pic(\Mod_{\A})(\Frnp) \cong \pi_0\Pic(\cMod_{\A})
\]
is thus a consequence of \Cref{lem:completepicgp} applied to $(\Mod_{(\E)_{\star}},(\E)_{\star}, (\E)_{\star}/(I_n)_{\star})$, since $(\E)_{\star}$ is complete with respect to $(\E)_{\star}/(I_n)_{\star}$. Finally, both groups are isomorphic to $\Z/2$ by \cite[Lemma 5.5]{ultra1}, which is essentially due to Baker and Richter \cite{bakerrichter_inv}. 
\end{proof}

\begin{rem}
Note that, by \Cref{prop:comparelocfun},  the localization functor $L_F$ on $\Frnp$ restricted to the category $\Mod_{\A}(\Frnp)$ is equivalent to the restriction of the functor $L_{\A \otimes F}$, which in turn transforms to Bousfield localization at $(\E)_{\star}/(I_n)_{\star}$ under Morita equivalence. 

For any $s \ge 1$, an argument similar to the one used in \Cref{prop:frtorsionmorita} shows that 
\[
\Psi_s(L_F(\A^{\otimes s})) \simeq (L_{\K}(\E^{\otimes s}))_{\star}.
\]
As an immediate consequence of \Cref{prop:etheory}, we see that the $\bE_{\infty}$-algebra $(L_{\K}(\E^{\otimes s}))_{\star}$ is even periodic and $\pi_0(L_{\K}(\E^{\otimes s}))_{\star}$ is complete with respect to the regular sequence $(p,u_1,\ldots,u_{n-1})$. However, we will not make use of these results in the remainder of this paper. 
\end{rem}

\section{Protoproducts} \label{ssec:four}

The category $\cSpnp$ has the property that $\cSpnp \simeq \Loc \Pic \cSpnp$ and it is compactly generated by the $\E$-localization of any finite type $n$ complex. However, unless $n=0$, the invertible objects in $\cSpnp$ are not compact and $\cSpnp$ is not equivalent to $\Ind \Thick \Pic \cSpnp$. This leads to various complications that we resolve by constructing a modified version of the protoproduct introduced originally in \cite{ultra1}.

\subsection{Modified Pic-generated protoproducts}\label{ssec:modpicproto}

Suppose $\cC$ is a compactly generated stable $\infty$-category and let $\cG \subseteq \cC$ be a collection of (not necessarily compact) generators of $\cC$ such that $\cC^{\omega} \subseteq \Thick(\cG)$. We filter the full subcategory $\Thick(\cG)$ of $\cC$ by $\cG$-cells up to retracts as in \cite[Section 3.5]{ultra1}, i.e., we define the $k$-th filtration category $\Cell_k$ to be the full subcategory of $\Thick(\cG)$ consisting of those objects which can be built in $k$ steps from $\cG$. We write $\Cell_*\Thick(\cG)$ for the resulting filtration and $(\cC^{\omega})_*$ and $(\Thick(\cG))_*$ for the constant filtrations on $\cC^{\omega}$ and $\Thick(\cG)$, respectively. 

By viewing these three filtrations as functors from $\N$ to $\Cat$, we may form a new filtration using the following homotopy pullback in $\Fun(\N,\Cat)$:
\[
\xymatrix{\cG\Cell_*(\cC^{\omega}) \ar[r] \ar[d] & \Cell_*\Thick(\cG) \ar[d] \\
(\cC^{\omega})_* \ar[r] & (\Thick(\cG))_*,}
\]
where the legs of the pullback diagram are the canonical inclusions. In other words, $\cG\Cell_*(\cC^{\omega})$ is the filtration obtained by intersecting the $\Cell_*$ filtration with the compact objects. We refer to $\cG\Cell_*(\cC^{\omega})$ as the $\cG$-cell filtration on $\cC^{\omega}$. Note that the filtrations $\cG\Cell_*(\cC^{\omega})$ and $(\cC^{\omega})_*$ in this diagram satisfy the conditions of a compact filtration given in \cite[Definition 3.31]{ultra1}. 

We now specialize to the case of interest in the remainder of the paper.

\begin{defn}\label{defn:picproto}
A set of objects $\cG \subset \cC$ is called a set of Pic-generators for a compactly generated presentably symmetric monoidal stable $\infty$-category $\cC$ if the following three conditions are satisfied:
	\begin{enumerate}
		\item $\cG$ is a subgroup of $\pi_0\Pic(\cC)$. In particular, the elements of $\cG$ are invertible.
		\item $\cG$ generates $\cC$ (i.e., $\Loc(\cG) = \cC$) and $\cC^{\omega} \subseteq \Thick(\cG)$. 
		\item $\cG$ is closed under suspensions and desuspensions.
	\end{enumerate}
When $\cG = \Pic(\cC)$, we denote the resulting $\cG$-cell filtration on $\cC^{\omega}$ by $\PicCell_*(\cC^{\omega})$, and call it the Pic-cell filtration. We say $\cC$ is Pic-generated if it admits a set of Pic-generators. 
\end{defn}

\begin{rem}
The unit in a compactly generated symmetric monoidal $\infty$-category is necessarily compact. The language above is somewhat unwieldly, but chosen carefully to distinguish from this case.
\end{rem}

\begin{defn}
Let $I$ be a set and let $(\cG_i)_{i \in I}$ be a collection of sets of Pic-generators for a collection of compactly generated presentably symmetric monoidal stable $\infty$-categories $(\cC_i)_{i \in I}$, i.e., for each $i\in I$, $\cG_i$ is a set of Pic-generators for $\cC_i$. The $\Pic$-generated protoproduct is defined to be
\[
\Prod{\cF}^{\Pic} (\cC_i,\cG_i) = \Prod{\cF}^{\flat}(\cC_i, \cG_i\Cell_*(\cC_i^{\omega})) = \Ind \Colim{d} \Prod{\cF}\cG_i\Cell_d(\cC_i^{\omega}).
\]
If $\cG_i = \Pic(\cC_i)$ for all $i$, then we write 
\[
\Prod{\cF}^{\Pic}\cC_i = \Prod{\cF}^{\Pic} (\cC_i,\PicCell_*(\cC_i^{\omega})).
\]
\end{defn}

\begin{ex}\label{ex:cellproto}
Besides the maximal example given by $\Pic(\cC)$ of a set of Pic-generators for a compactly generated presentably symmetric monoidal stable $\infty$-category $\cC$, there is also a minimal one, namely the set $\cG_{\cC}^{\cell}$ of all shifts of the unit of $\cC$. Let $\cG_{\cC}^{\cell}\Cell_*(\cC_i^{\omega})$ be the associated filtration. Let $(\cG_i^{\cell}=\cG_{\cC_i}^{\cell})_{i \in I}$ be a collection of sets of such generators for a collection of compactly generated presentably symmetric monoidal stable $\infty$-categories $(\cC_i)_{i \in I}$. As an instance of \Cref{defn:picproto}, we define the cell protoproduct as 
\[
\Prod{\cF}^{\flat}\cC_i = \Prod{\cF}^{\Pic} (\cC_i,\cG_i^{\cell}) = \Prod{\cF}^{\flat}(\cC_i, \cG_{i}^{\cell}\Cell_*(\cC_i^{\omega})) = \Ind \Colim{d} \Prod{\cF}\cG_{i}^{\cell}\Cell_d(\cC_i^{\omega}).
\]
This choice of notation is justified by the observation that in the case where the unit of $\cC_i$ is compact for almost all $i \in I$, the construction of the cell protoproduct specializes to the protoproduct as defined in \cite[Section 3.5]{ultra1}.
\end{ex}

Our goal will be to study these protoproducts. In particular, we will show that the protoproduct is symmetric monoidal and Pic-generated under a uniformity condition.

\begin{convention}
For the rest of this section and unless specified otherwise, $(\cC_i)_{i \in I}$ and $(\cD_i)_{i \in I}$ are collections of compactly generated presentably symmetric monoidal stable $\infty$-categories.
\end{convention}

\subsection{Multiplicative properties of $\Pic$-generated protoproducts}

The goal of this subsection is to study the monoidal properties of the $\Pic$-generated protoproduct construction for symmetric monoidal $\infty$-categories in which the unit is not necessarily compact. In order to find suitable conditions on a collection of $\infty$-categories to guarantee the existence of a monoidal unit in the $\Pic$-generated protoproduct, we need to work with a non-unital version $\bE_{\infty}^{\text{nu}}$ of the $\bE_{\infty}$ operad introduced and studied by Lurie in \cite[Section 5.4.4]{ha}.

\begin{defn}
A symmetric oidal $\infty$-category is an object in the $\infty$-category $\Alg_{\bE_{\infty}^{\text{nu}}}(\mathrm{Cat}_{\infty}^{\omega})$ and we refer to a map in this $\infty$-category as a symmetric oidal functor. If $\cC = (\cC,\otimes)$ is a symmetric oidal $\infty$-category, then the adjoint to $\otimes\colon \cC \times \cC \to \cC$ is a functor $\cC \to \Fun(\cC,\cC)$. By construction, this functor lands in the subcategory of colimit preserving functors and we refer to the resulting functor 
\[
\xymatrix{\Cayley\colon \cC \ar[r] & \Fun^L(\cC,\cC),}
\]
as the Cayley functor. Informally speaking, $\Cayley$ sends an object $x \in \cC$ to the colimit preserving functor $x \otimes -\colon \cC \to \cC$. 
\end{defn}

Just as there is a notion of a symmetric monoidal filtration of \cite[Definition 3.33]{ultra1}, there is an analogous notion of a symmetric oidal filtration on the compact objects in a symmetric monoidal $\infty$-category. Note that if $\cG$ is a set of $\Pic$-generators for a compactly generated presentably symmetric monoidal stable $\infty$-category $\cC$, then $\cG\Cell_*(\cC^{\omega})$ is a symmetric oidal filtration.

\begin{lem}\label{lem:nsymmmon}
Let $(\cG_i)_{i \in I}$ be a collection of sets of Pic-generators for $(\cC_i)_{i \in I}$, then the Pic-generated protoproduct
\[
\Prod{\cF}^{\Pic}(\cC_i,\cG_i)
\]
is a full symmetric oidal subcategory of $\Prod{\cF}^{\omega} \cC_i$ with lax symmetric oidal right adjoint $n$.
\end{lem}
\begin{proof}
Lemma 3.37 in \cite{ultra1} gives fully faithfullness. Note that $\Prod{\cF}^{\omega} \cC_i$ is symmetric oidal. Also, $\Prod{\cF}^{\Pic}(\cC_i,\cG_i)$ is a full subcategory of $\Prod{\cF}^{\omega} \cC_i$. Thus, to check that $\Prod{\cF}^{\Pic}(\cC_i,\cG_i)$ is symmetric oidal, we just need to check the tensor product on $\Prod{\cF}^{\omega} \cC_i$ restricts to $\Prod{\cF}^{\Pic}(\cC_i,\cG_i)$. But this follows from the fact that the filtration is symmetric oidal. The right adjoint is naturally promoted to a lax symmetric oidal functor by \cite[Corollary 7.3.2.7]{ha}.
\end{proof}

When $(\cC_i)_{i \in I}$ and $(\cG_i)_{i \in I}$ satisfy the conditions of the lemma, we will write $\Cayley_{\cF}$ for the Cayley functor associated to the symmetric oidal category $\Prod{\cF}^{\Pic}(\cC_i,\cG_i)$. 

\begin{cor} \label{cor:symmonfun}
Let $(\cG_i)_{i \in I}$ and $(\cH_i)_{i \in I}$ be collections of sets of Pic-generators for $(\cC_i)_{i \in I}$ and $(\cD_i)_{i \in I}$, respectively, let $(f_i)_{i \in I}$ be symmetric monoidal functors (necessarily preserving compact objects) such that $f_i$ takes $\cG_i$ to $\cH_i$. The protoproduct
\[
\xymatrix{\Prod{\cF}^{\Pic}(\cC_i,\cG_i) \ar[r]^{\Prod{\cF}^{\Pic}f_i} & \Prod{\cF}^{\Pic}(\cD_i,\cH_i)}
\]
is a symmetric oidal functor with lax symmetric oidal right adjoint $g$.
\end{cor}

Our next goal is to establish conditions on a collection $(\cG_i)_{i \in I}$ of sets of Pic-generators that guarantee the existence of a unit in $\Prod{\cF}^{\Pic}(\cC_i,\cG_i)$. We begin with the definition of a quasi-unit.

\begin{defn}
An object $u$ in a symmetric oidal $\infty$-category $\cC$ is called a quasi-unit if there exists a natural equivalence
\[
\Cayley(u) \simeq u \otimes - \implies \Id_{\cC}
\]
of endofunctors on $\cC$. 
\end{defn}

In light of the following result due to Lurie~\cite[Corollary 5.4.4.7]{ha}, the problem reduces to the construction of a quasi-unit in $\Prod{\cF}^{\Pic}(\cC_i,\cG_i)$.

\begin{lem}\cite[Corollary 5.4.4.7]{ha}\label{lem:luriequasiunit}
A symmetric oidal $\infty$-category $\cC$ can be promoted to a symmetric monoidal $\infty$-category in a unique way if it contains a quasi-unit. A symmetric oidal functor $\cC \to \cD$ between symmetric monoidal $\infty$-categories can be promoted to a symmetric monoidal functor in a unique way if it send a quasi-unit to a quasi-unit. 
\end{lem}

Let $J$ be a filtered diagram and let $\cC$ be a small $\infty$-category. Consider a functor $f \colon J^{\rhd} \to \cC$. We will call $f$ a formal colimit cone if the composite
\[
\xymatrix{J^{\rhd} \ar[r]^-{f} & \cC \ar[r] & \Ind \cC}
\] 
is a colimit cone. Note that the notion of a formal colimit is stronger than the notion of a colimit cone; in particular all formal colimit cones are colimit cones. If $\cC$ is pointed and $f \colon J^{\rhd} \to \cC$ is a functor sending the cone point to the $0$-object in $\cC$, then $f$ is a formal colimit cone if and only if for each $n \in J$ there exists an $m \in J$ and a map $l \colon n \to m$ in $J$ such that $f(l)$ is null. 

For the remainder of this subsection, let $I$ be a set and let $\cF$ be an ultrafilter on $I$.

Recall that the join $\star$ of simplicial sets is built using finite products and coproducts and thus commutes with the ultraproduct. Therefore, there is an isomorphism of simplicial sets $\Prod{\cF} (A_i \star B_i) \cong (\Prod{\cF}A_i) \star (\Prod{\cF} B_i)$, which specializes to an isomorphism of diagrams
\begin{equation}\label{eq:ultracone}
\Prod{\cF} (J_{i}^{\rhd}) \cong (\Prod{\cF} J_i)^{\rhd}.
\end{equation}

\begin{lem} \label{hey}
Assume that $G \colon \cC \to \cD$ is a fully faithful finite colimit preserving functor between $\infty$-categories with finite colimits and assume that $J$ is a filtered diagram then $f \colon J^{\rhd} \to \cC$ is a formal colimit cone if and only if $G \circ f$ is a formal colimit cone. 
\end{lem}
\begin{proof}
This follows from the fact that $\Ind(G)$ is a fully faithful colimit preserving functor.
\end{proof}

\begin{lem} \label{formalco}
Let $(J_i)_{i \in I}$ be a collection of filtered diagram categories and let $(\cC_i)_{i \in I}$ be a collection of small stable $\infty$-categories. Assume that, for each $i \in I$, we have a formal colimit cone $f_i \colon J_{i}^{\rhd} \to \cC_i$. Then
\[
\xymatrix{\Prod{\cF} J_{i}^{\rhd} \ar[r] &  \Prod{\cF} \cC_i}
\]
is a formal colimit cone.
\end{lem}
\begin{proof}
We must show that the induced map 
\[
\xymatrix{\Prod{\cF} f_i \colon (\Prod{\cF} J_i)^{\rhd} \ar[r] & \Prod{\cF}\cC_i}
\]
is a formal colimit. By stability, we can reduce to the case that each functor $f_i \colon J_{i}^{\rhd} \to \cC_i$ sends the cone point to the zero object. Thus the cone point in $(\Prod{\cF} J_i)^{\rhd}$ is sent to the zero object in $\Prod{\cF}\cC_i$ by $\Prod{\cF} f_i$. It now suffices to check that for each $[n_i] \in \Prod{\cF} J_i$ there exists an object $[m_i] \in \Prod{\cF} J_i$ and a map $l \colon [n_i] \to [m_i]$ in $\Prod{\cF} J_i$ such that $(\Prod{\cF}f_i)(l)$ is null. We may choose $l$ to be $[l_i]$, where $l_i$ has the corresponding property for $J_i$.
\end{proof}

Let $(\cC, F_*)$ and $(\cD, G_*)$ be compactly generated $\infty$-categories with filtrations. Let $\beta \colon \N \to \N$ be a non-decreasing function. We define $\Fun^{\beta}( (\cC, F_*), (\cD, G_*))$ to be the full subcategory of $\Catomega(\cC, \cD)$ on the functors $f \colon \cC \to \cD$ such that
\[
\xymatrix{f^{\omega} \colon \cC^{\omega} \ar[r] & \cD^{\omega}}
\]
has the property that if $c \in F_{k}\cC$ then $f^{\omega}(c) \in G_{\beta(k)}\cD$. 

\begin{lem}\label{lem:triangle}
Let $(\cC_i, F_{i,*})_{i \in I}$ be a symmetric oidal collection of compactly generated symmetric oidal $\infty$-categories and let $k$ be a fixed natural number. There exists a non-decreasing function $\beta_k \colon \N \to \N$ such that the restriction of the Cayley map $\Cayley_{\cF}$ to $\Prod{\cF}F_{i,k}\cC_i$ factors as indicated in the following commutative diagram:
\[
\xymatrix{\Prod{\cF}F_{i,k}\cC_i \ar[r] \ar[d] & \Prod{\cF}^{\flat}(\cC_i,F_{i,*}) \ar[d]^{\Cayley_{\cF}}  \\ 
\Prod{\cF} \Fun^{\beta_k}((\cC_i, F_{i,*}), (\cC_i, F_{i,*})) \ar@{-->}[r] & \Fun^L(\Prod{\cF}^{\flat}(\cC_i,F_{i,*}), \Prod{\cF}^{\flat}(\cC_i,F_{i,*})).}
\]
\end{lem}
\begin{proof}

Fix a non-decreasing function $\beta \colon \N \to \N$, then for any two subsets $V \subseteq U$ in $\cF$ and $s \ge 0$, there are natural coordinate-wise evaluation functors 
\[
\xymatrix{\Prod{i \in V} F_{i, s} \cC_i \times \Prod{i \in U}\Fun^{\beta}((\cC_i, F_{i,*}), (\cD_i, G_{i,*})) \ar[r] & \Prod{i \in V} G_{i, \beta(s)} \cD_i.}
\]
Note that the functors in $\Fun^{\beta}((\cC_i, F_{i,*}), (\cD_i, G_{i,*}))$ preserve all colimits. Passing to the colimit first over all $V \in \cF$ contained in $U$ and then over $s\ge 0$ thus induces a functor
\[
\xymatrix{\Ind\colim_{s} \Prod{\cF} F_{i, s} \cC_i \times \Prod{i \in U}\Fun^{\beta}((\cC_i, F_{i,*}), (\cD_i, G_{i,*})) \ar[r] & \Ind \colim_{s} \Prod{\cF} G_{i, \beta(s)} \cD_i.}
\]
By naturality in $U$ and the definition of the protoproduct, varying $U$, and using adjunction then yields a functor
\begin{equation}\label{eq:betafun}
\xymatrix{\Prod{\cF}\Fun^{\beta}((\cC_i, F_{i,*}), (\cD_i, G_{i,*})) \ar[r] & \Fun^L(\Prod{\cF}^{\flat}(\cC_i,F_{i,*}), \Prod{\cF}^{\flat}(\cD_i,G_{i,*})).}
\end{equation}

Returning to the specific situation of the lemma, we first observe that for any $i \in I$ and fixed $k$, the Cayley map $\Cayley_i$ for $\cC_i$ factors as follows
\[
\xymatrix{F_{i,k}\cC_i \ar[r] \ar@{-->}[d]_{\Cayley_{i,k}} & \cC_i \ar[d]^{\Cayley_i} \\
\Fun^{\beta_k}((\cC_i,F_{i,*}),(\cC_i,F_{i,*})) \ar[r] & \Fun^L(\cC_i,\cC_i),}
\]
with $\beta_k$ independent of $i$ by assumption. Let $V \subseteq U$ be subsets in $\cF$. Composing the functor $\prod_{i \in V}\mathrm{id}_{F_{i, s}\cC_i} \times \prod_{i \in U} \Cayley_{i,k}$ with the evaluations, we thus obtain natural functors 
\[
\xymatrix{\prod_{i \in V} F_{i, s} \cC_i \times \prod_{i \in U}F_{i,k}\cC_i \ar[r] & \prod_{i \in V} F_{i, s} \cC_i \times \prod_{i \in U}\Fun^{\beta_k}((\cC_i,F_{i,*}),(\cC_i,F_{i,*})) \ar[d] \\ 
& \prod_{i \in V}F_{i, \beta_k(s)}\cC_i.}
\]
Informally speaking, this composite is given in each coordinate $i \in V$ by the monoidal product in $\cC_i$ restricted to the corresponding filtration step, i.e., the functor $\otimes \colon F_{i, s}\cC_i \times F_{i,k}\cC_i \to F_{i, \beta_k(s)}\cC_i$. Passing to colimits and unwinding the adjunction as in the construction of the functor in \eqref{eq:betafun} gives the desired factorization. 
\end{proof}

\begin{defn}\label{def:uniformunital}
A collection $(\cG_i)_{i \in I}$ of sets of Pic-generators for $(\cC_i)_{i \in I}$ is called unital if there is a natural number $d$ such that for all $i \in I$ the $\infty$-category $\Ind(\cG_i\Cell_d(\cC_i^{\omega})) \subset \cC_i$ contains the unit of $\cC_i$. In the special case when $(\cG_i)_{i \in I} = (\Pic(\cC_i))_{i\in I}$, we will refer to $(\cC_i)_{i \in I}$ as a unital collection of Pic-generated $\infty$-categories.

A collection of functors $(f_i)_{i \in I}$ between two such collections $(\cC_i,\cG_i)_{i \in I}$ and $(\cD_i,\cH_i)_{i \in I}$ is by definition a set of symmetric monoidal functors $f_i\colon \cC_i \to \cD_i$ that preserve colimits and compact objects.
\end{defn}

\begin{prop}\label{prop:picprotomonoidal}
For any unital collection $(\cG_i)_{i\in I}$ of sets of Pic-generators for $(\cC_i)_{i \in I}$ and any ultrafilter $\cF$ on $I$, the category $\Prod{\cF}^{\Pic}(\cC_i,\cG_i)$ contains a quasi-unit.
\end{prop}
\begin{proof}
The definition of a unital collection of Pic-generated $\infty$-categories implies that there exists a $d \in \N$ and functors $f_i \colon J_i \to \cG_i\Cell_d(\cC_i^{\omega})$ for each $i \in I$ such that $J_i$ is filtered and $f_i$ picks out the unit in $\cC_{i}$. In \cite[Proposition 3.19]{ultra1}, we show that $\Prod{\cF}J_i$ is filtered. Let 
\[
u = \colim(\Prod{\cF}J_i \lra{\Prod{\cF}f_i} \Prod{\cF}\cG_i\Cell_d(\cC_i^{\omega}) \to \Prod{\cF}^{\Pic}(\cC_i,\cG_i)).
\] 
We will show that $u$ is a quasi-unit for $\Prod{\cF}^{\Pic}(\cC_i,\cG_i)$. First we will construct a natural transformation
\[
u \otimes - \implies \Id.
\]
For each $\cC_i$, the functor $f_i \colon J_i \to \cG_i\Cell_d(\cC_i^{\omega}) \subset \cC_i$ has a colimit, giving us functors $f_{i}^{\rhd} \colon J_{i}^{\rhd} \to \cC_i$. Now consider the composite
\[
\xymatrix{g_i \colon J_{i}^{\rhd} \ar[r]^-{{f_{i}^{\rhd}}} & \cC_i \ar[r]^-{\Cayley_i} & \Fun(\cC_i, \cC_i).}
\]
By the oidality assumption on the filtration, there exists a function $\beta_d\colon \N \to \N$ independent of $i$ and a diagram (not including the lifting)
\[
\xymatrix{J_i \ar[r]^-{f_i} \ar[d] & \cG_i\Cell_d(\cC_i^{\omega}) \ar[r] & \Fun^{\beta_d}((\cC_{i},\cG_i\Cell_*(\cC_i^{\omega})), (\cC_{i},\cG_i\Cell_*(\cC_i^{\omega}))) \ar[d] \\ J_{i}^{\rhd} \ar@{-->}[urr] \ar[rr]_-{g_i} &  & \Fun(\cC_i,\cC_i)}
\]
for each $i$. We will show that the lift exists. Since the right vertical arrow is fully faithful, it is enough to provide a lift on objects. For the objects in $J_i$, this is clear. Since the cone point goes to the identity functor in $\Fun(\cC_i,\cC_i)$, it lifts as well. Since the Cayley functor is a colimit preserving functor, the dashed arrow is a colimit diagram.

By applying the ultraproduct to the above diagrams and using \Cref{lem:triangle}, we get the following commutative diagram:
\[
\resizebox{\textwidth}{!}{
\xymatrix{\Prod{\cF} J_i \ar[r] \ar[d] & \Prod{\cF}\cG_i\Cell_d(\cC_i^{\omega}) \ar[r] \ar[d] & \Prod{\cF}^{\Pic} (\cC_i,\cG_i) \ar[d]  \\ \Prod{\cF} J_{i}^{\rhd} \ar[r] & \Prod{\cF} \Fun^{\beta_d}((\cC_{i},\cG_i\Cell_d(\cC_i^{\omega})), (\cC_{i},\cG_i\Cell_d(\cC_i^{\omega}))) \ar[r] & \Fun^L(\Prod{\cF}^{\Pic} (\cC_i,\cG_i), \Prod{\cF}^{\Pic} (\cC_i,\cG_i).}
}
\]
Recall that $\Prod{\cF} J_{i}^{\rhd} \cong (\Prod{\cF} J_{i})^{\rhd}$ by Equation \eqref{eq:ultracone}. Consider the composition $h$ of the horizontal bottom arrows which sends the cone point to the identity functor. Since the tensor product commutes with colimits in each variable, the colimit of $h$ restricted to $\Prod{\cF} J_i$ is equivalent to $u \otimes -$. Therefore, to show that $u$ is a quasi-unit, it is enough to show that $h$ is a colimit diagram.

To this end, let $X = [X_i] \in \Prod{\cF} \cG_i\Cell_k(\cC_i^{\omega})$, we can extend the diagram above to 
\[
\resizebox{\textwidth}{!}{
\xymatrix{\Prod{\cF} J_i \ar[r] \ar[d]  & \Prod{\cF}\cG_i\Cell_d(\cC_i^{\omega}) \ar[r] \ar[d] & \Prod{\cF}^{\Pic} (\cC_i,\cG_i) \ar[d]  \\ 
\Prod{\cF} J_{i}^{\rhd} \ar[rd]^{G_X} \ar[rdd]_{H_X} \ar[r] & \Prod{\cF} \Fun^{\beta_d}((\cC_{i},\cG_i\Cell_*(\cC_i^{\omega})), (\cC_{i},\cG_i\Cell_*(\cC_i^{\omega}))) \ar[r] \ar[d]_{\text{ev}_X} & \Fun^L(\Prod{\cF}^{\Pic} (\cC_i,\cG_i), \Prod{\cF}^{\Pic} (\cC_i,\cG_i)) \ar[dd]^{\text{ev}_{X}'} \\ 
& \Prod{\cF} \cG_i\Cell_{\beta_d(k)}(\cC_i^{\omega}) \ar[d] & \\
& (\Prod{\cF}^{\Pic} (\cC_i,\cG_i))^{\omega} \ar[r] & \Prod{\cF}^{\Pic}(\cC_i,\cG_i),}
}
\]
where $\text{ev}_X$ and $\text{ev}_{X}'$ are the evaluation maps at $X$. Since colimits in functor categories are computed pointwise and because the objects of $\prod_{\cF} \cG_i\Cell_k(\cC_i^{\omega})$ as $k$ varies form a set of generators for $\prod_{\cF}^{\Pic} (\cC_i,\cG_i)$, to check that $h$ is a colimit diagram, it suffices to check that $\text{ev}_{X}' \circ h$ is a colimit diagram for each such $X$. Therefore, we have reduced the claim to showing that $H_X$ is a formal colimit for all $X$. \Cref{hey} implies that the map 
\[
\xymatrix{\Prod{\cF}\cG_i\Cell_{\beta_d(k)}(\cC_i^{\omega}) \ar[r] & (\Prod{\cF}^{\Pic} (\cC_i,\cG_i))^{\omega}}
\]
preserves formal colimits. Therefore it is enough to show that $G_X$ is a formal colimit cone. The functor $f_{i}^{\rhd}$ defined above is a colimit diagram. Define $G_{X_i} = (- \otimes X_i) \circ f_{i}^{\rhd}$. The functor $G_{X_i}$ factors through $\cG_i\Cell_{\beta_d(k)}(\cC_i^{\omega})$, and we will take this to be the target of $G_{X_i}$. By \Cref{hey}, $G_{X_i}$ is a formal colimit cone and since $G_X = \Prod{\cF}G_{X_i}$, by  \Cref{formalco} we are done. 
\end{proof}

\begin{cor}\label{cor:picprotomonoidal}
For any unital collection $(\cG_i)_{i\in I}$ of sets of Pic-generators for $(\cC_i)_{i \in I}$ and any ultrafilter $\cF$ on $I$, the category $\Prod{\cF}^{\Pic}(\cC_i,\cG_i)$ is equipped with a canonical symmetric monoidal structure. Moreover, there is a canonical symmetric monoidal functor
\[
\xymatrix{\Prod{\cF}^{\Pic}(\cC_i,\cG_i) \ar[r] & \Prod{\cF}^{\Pic}(\cC_i,\Pic(\cC_i)) = \Prod{\cF}^{\Pic}\cC_i}
\]
with lax symmetric monoidal right adjoint.
\end{cor}
\begin{proof}
By \Cref{lem:nsymmmon} and \Cref{prop:picprotomonoidal}, $\Prod{\cF}^{\Pic}(\cC_i,\cG_i)$ is a symmetric oidal $\infty$-category equipped with a quasi-unit, so the first claim is a consequence of \Cref{lem:luriequasiunit}, while the second follows from this and \Cref{cor:symmonfun}. 
\end{proof}

We record another corollary.

\begin{cor}\label{cor:picprotofun}
Let $(f_i\colon \cC_i \to \cD_i)_{i\in I}$ be a collection of functors between unital collections of Pic-generated $\infty$-categories (see \cref{def:uniformunital}). For any ultrafilter $\cF$ on $I$, the collection $(f_i)_{i\in I}$ induces a colimit preserving symmetric monoidal functor  
\[
\xymatrix{\Prod{\cF}^{\Pic}\cC_i \ar[r] & \Prod{\cF}^{\Pic}\cD_i.}
\]
\end{cor}
\begin{proof}
The definition of functors between unital Pic-generated $\infty$-categories guarantees that the $\Pic$-cell filtrations and the quasi-unit are preserved, hence by \Cref{cor:symmonfun} and \Cref{prop:picprotomonoidal} we obtain an induced functor between $\Pic$-protoproducts with the desired properties.
\end{proof}

The previous two corollaries admit a common generalization, which will be required in the construction of the comparison functor in \Cref{ssec:comparison}. The next lemmas will be useful in the proof.

\begin{lem}\label{lem:funfactory}
Let $I$ be a set and $\cF$ an ultrafilter on $I$. Further, let $(\cA_i)_{i\in I}$ and $(\cB_i)_{i \in I}$ be two collections of $\infty$-categories. Any collection of functors $(f_i\colon \cA_i \to \Ind\cB_i)_{i \in I}$ induces a functor
\[
\xymatrix{\Prod{\cF}\cA_i \ar[r] & \Ind \Prod{\cF}\cB_i}
\]
which, informally speaking, can be described as follows: Given $a_i \in \cA_i$, represent $f_i(a_i)$ by a filtered diagram $\cK_i \to \cB_i$. The image of $[a_i]_{\cF}$ is then the filtered diagram $\Prod{\cF}f_i(a_i)\colon  \Prod{\cF}\cK_i \to \Prod{\cF}\cB_i$. Moreover, if $(f_i\colon \cA_i \to \Ind\cB_i) \in \Alg_{\bE_{\infty}^{\text{nu}}}(\mathrm{Cat}_{\infty})$ for all $i \in I$, then we obtain a symmetric oidal functor
\[
\xymatrix{\Ind \Prod{\cF} \cA_i \ar[r] & \Ind \Prod{\cF} \cB_i}
\] 
by extending to the ind-category in the source.
\end{lem}
\begin{proof}
The functor in the statement of the lemma is the composite
\[
\xymatrix{\Prod{\cF}\cA_i \ar[r] & \Prod{\cF} \Ind \cB_i \ar[r]^{m} & \Ind \Prod{\cF}\cB_i},
\]
where $m$ is the map constructed at the beginning of \cite[Section 3.3]{ultra1}. The first map is symmetric oidal as it is an ultraproduct of symmetric oidal functors.  The second functor is symmetric oidal functor by the proof of \cite[Corollary 3.26]{ultra1}, ignoring Condition (4) (the unit condition) of \cite{AFT} that plays a role in \cite[Lemma 3.25]{ultra1}. 
\end{proof}

\begin{prop}\label{prop:superfunctors}
Suppose $(\cG_i)_{i\in I}$ and $(\cH_i)_{i\in I}$ are unital collections of sets of Pic-generators for $(\cC_i)_{i \in I}$ and $(\cD_i)_{i \in I}$, respectively. Let $(f_i)_{i \in I}$ be a collection of functors $f_i \colon \cC_i \to \cD_i$ between stable $\infty$-categories satisfying the following properties:
	\begin{enumerate}
		\item For each $i \in I$, the functor $f_i$ is symmetric monoidal and preserves colimits.
		\item There exists a strictly increasing function $\beta\colon \N \to \N$ such that for all $i \in I$ and all $k \ge 0$, the functor $f_i$ restricts to a functor
		\[
		\xymatrix{f_{i,k}\colon \cG_i\Cell_k(\cC_i^{\omega}) \ar[r] & \Ind\cH_i\Cell_{\beta(k)}(\cD_i^{\omega}).}
		\]
	\end{enumerate}
For any ultrafilter $\cF$ on $I$, there exists a colimit preserving symmetric oidal functor 
\[
\xymatrix{\Prod{\cF}^{\Pic}f_i\colon \Prod{\cF}^{\Pic}(\cC_i,\cG_i) \ar[r] & \Prod{\cF}^{\Pic}(\cD_i,\cH_i).}
\]
\end{prop}
\begin{proof}
We start with the construction of the desired functor. For fixed $k \ge 0$, we may apply \Cref{lem:funfactory} to the collection $(f_{i,k})_{i \in I}$ to obtain a functor
\[
\xymatrix{\Prod{\cF}\cG_i\Cell_k(\cC_i^{\omega}) \ar[r] & \Ind\Prod{\cF}\cH_i\Cell_{\beta(k)}(\cD_i^{\omega}).}
\]
Extending this functor to the ind-category of the source and passing to colimits over $k$ in $\mathrm{Cat}_{\infty}^{\omega}$ then yields a colimit preserving functor
\[
\Prod{\cF}^{\Pic}(\cC_i,\cG_i) \simeq  \Colim{k}\Ind \Prod{\cF}\cG_i\Cell_k(\cC_i^{\omega}) \to\Colim{k}  \Ind  \Prod{\cF}\cH_i\Cell_{\beta(k)}(\cD_i^{\omega}) \simeq \Prod{\cF}^{\Pic}(\cD_i,\cH_i)
\]
which we denote by $\Prod{\cF}^{\Pic}f_i$. 

It remains to verify the symmetric monoidal properties of this functor. To this end, first observe that $\Prod{\cF}^{\Pic}f_i$ fits into a commutative square
\[
\xymatrix{ \Prod{\cF}^{\Pic}(\cC_i,\cG_i) \ar[r]^-{\Prod{\cF}^{\Pic}f_i} \ar[d] & \Prod{\cF}^{\Pic}(\cD_i,\cH_i) \ar[d] \\
\Prod{\cF}^{\omega}\cC_i \ar[r] &  \Prod{\cF}^{\omega}\cD_i,}
\]
in which the bottom functor is symmetric oidal by \Cref{lem:funfactory} and the vertical arrows are fully faithful and symmetric oidal by \Cref{lem:nsymmmon}. This implies that $\Prod{\cF}^{\Pic}f_i$ is symmetric oidal as well. 
\end{proof}

\begin{cor}
Under the assumptions of \cref{prop:superfunctors}, if additionally the unit in $\cC_i$ is compact, then $\Prod{\cF}^{\Pic}f_i$ is unital.
\end{cor}
\begin{proof}
Let $u_i \in \cC_{i}^{\omega}$ be the unit. Consider the commutative diagram
\[
\xymatrix{\ast \ar[r]^-{u_i} & \cG_i\Cell_1(\cC_i^{\omega}) \ar[r] \ar[d] & \Ind\cH_i\Cell_{\beta(1)}(\cD_i^{\omega}) \ar[d] \\
& \cC_i \ar[r]_-{f_i} & \cD_i,}
\]
where the vertical functors are fully faithful. Since $f_i$ is unital, the top composite thus gives a representation of the unit in $\cD_i$. Applying ultraproducts to this diagram and the map in \cref{lem:funfactory} gives the following commutative diagram
\[
\xymatrix{\ast \ar[r]^-{\Prod{\cF}u_i} \ar[rd]_u & \Prod{\cF}\cG_i\Cell_1(\cC_i^{\omega}) \ar[r] \ar[d] & \Ind\Prod{\cF}\cH_i\Cell_{\beta(1)}(\cD_i^{\omega}) \ar[d] \\
& \Prod{\cF}^{\Pic}(\cC_i,\cG_i) \ar[r]_-{\Prod{\cF}^{\Pic}f_i} & \Prod{\cF}^{\Pic}(\cD_i,\cH_i).}
\]
The proof of \Cref{prop:picprotomonoidal} and the description of the functor in \cref{lem:funfactory} shows that the composite along the top is a quasi-unit for $\Prod{\cF}^{\Pic}(\cD_i,\cH_i)$. In other words, we see that $\Prod{\cF}^{\Pic}f_i$ preserves quasi-units, so \Cref{lem:luriequasiunit} furnishes the claim.
\end{proof}

\subsection{$\Pic$-generated protoproducts are $\Pic$-generated}

Our next goal is to prove that under the assumptions that $(\cC_i)_{i\in I}$ is a unital collection of Picard-generated $\infty$-categories in the sense of \Cref{def:uniformunital}, the Pic-generated protoproduct $\Prod{\cF}^{\Pic}\cC_i$ is in fact generated by its Picard $\infty$-groupoid.

\begin{lem} \label{superlemma}
Assume $\cC$ is a symmetric monoidal $\infty$-category. Then $\Pic(\cC) \otimes \cC^{\omega} = \cC^{\omega}$.
\end{lem}
\begin{proof}
Let $X \in \Pic(\cC)$, then $X \otimes -$ is an equivalence of $\infty$-categories. Since an equivalence sends compact objects to compact objects, we see that $X \otimes W$ is compact for any compact $W \in \cC^{\omega}$.
\end{proof}

In the proof of the next proposition we will make use of $(\mathrm{Cat}_{\infty}^{\textrm{perf}}, \square, \Sp^{\omega})$, the symmetric monoidal $\infty$-category of small idempotent complete stable $\infty$-categories and exact functors (see \cite[Section 3.1]{bgt} for more details). We will also make use of $(\mathrm{Cat}_{\infty}^{\omega, \mathrm{st}}, \boxtimes, \Sp)$, the $\infty$-category of compactly generated stable $\infty$-categories and colimit and compact object preserving maps. These symmetric monoidal structures are closely related: given $\cC$ and $\cD$ in $\mathrm{Cat}_{\infty}^{\omega, \mathrm{st}}$, we have
\[
\cC \boxtimes \cD \simeq \Ind(\cC^{\omega} \square \cD^{\omega}).
\]

\begin{prop}\label{prop:picproto}
Let  $(\cC_i)_{i\in I}$ be a unital collection of Pic-generated $\infty$-categories  and $\cF$ an ultrafilter on $I$, then $\Prod{\cF}^{\Pic}\cC_i$ is a unital Pic-generated $\infty$-category, i.e., the canonical inclusion functor
\[
\xymatrix{\Loc\Pic\Prod{\cF}^{\Pic}\cC_i \ar[r]^-{\sim} & \Prod{\cF}^{\Pic}\cC_i}
\]
is an equivalence of symmetric monoidal $\infty$-categories.
\end{prop}
\begin{proof}
Throughout this proof, we write $\cC = \Prod{\cF}^{\Pic}\cC_i$ and $\CC$ for the Pic-generated protoproduct $\Prod{\cF}^{\Pic}\Ind\Thick\Pic(\cC_i)$ of the compactly Pic-generated $\infty$-categories $\Ind\Thick\Pic(\cC_i)$. In order to show that $\Pic\Prod{\cF}^{\Pic}\cC_i$ provides a collection of  generators for $\cC$, we will construct a symmetric monoidal functor
\[
\xymatrix{\cL\colon \CC  = \Prod{\cF}^{\Pic}\Ind\Thick\Pic(\cC_i) \ar[r] & \Prod{\cF}^{\Pic}\cC_i = \cC}
\]
with a conservative right adjoint $\cR$. It follows that $\cL(\Pic(\CC))$ forms a collection of generators for $\cC$ (\cite[Lemma 2.25]{arthur}). Because $\cL$ is symmetric monoidal, $\cL(\Pic(\CC)) \subseteq \Pic(\cC)$. Moreover, given $X \in \cC^{\omega}$, there exists a natural number $d$ such that 
\[
X \in \Prod{\cF}\PicCell_d(\cC_i^{\omega}) \subseteq \Prod{\cF}\Cell_d\Thick(\Pic(\cC_i)).
\] 
Therefore, $X \in \Thick(\Pic(\cC))$, hence $\cC^{\omega} \subseteq \Thick(\Pic(\cC))$. This proves the proposition. 

We begin with the construction of $\cL$. Observe that $\CC$ is a symmetric monoidal stable $\infty$-category compactly generated by its Picard $\infty$-groupoid and in particular has a compact unit $\mathbf{1}_{\CC}$ (see \cite[Corollary 3.59]{ultra1}). For any $s\ge 0$, the construction of the Pic-cell filtration comes with an inclusion
\[
\xymatrix{\PicCell_{i,s}\cC_i \ar[r] & \Cell_{i,s}\Thick(\Pic\cC_i)}
\]
that extends to a fully faithful and oidal functor $\iota\colon \cC^{\omega} \to \CC^{\omega}$ as ultraproducts preserve fully faithfulness of functors by \cite[Corollary 3.15]{ultra1}. The first step in the construction of $\cL$ is to produce a factorization as indicated in the next diagram:
\begin{equation}\label{eq:action}
\xymatrix{\Sp^{\omega} \square \cC^{\omega} \ar[r]^-{\mathbf{1}_{\CC}\square \id} & \CC^{\omega} \square \cC^{\omega} \ar@{-->}[r]^-{\cA^{\omega}} \ar[d]_{\id\square \iota} & \cC^{\omega} \ar[d]^{\iota} \\
& \CC^{\omega} \square \CC^{\omega} \ar[r]_-{\otimes} & \CC^{\omega}.}
\end{equation}
Here, we have written $\otimes$ for the symmetric monoidal structure of $\CC$. Furthermore, $\mathbf{1}_{\CC}$ denotes the canonical finite colimit preserving functor $\Sp^{\omega} \to \CC$ which sends $S^0$ to $\mathbf{1}_{\CC}$.

To see that the desired factorization $\cA^{\omega}$ in $\eqref{eq:action}$ exists, it suffices to check the claim objectwise since $\iota$ is fully faithful. Indeed, unwinding the construction of the $\Pic$-cell filtration, this reduces to the fact that $\Pic(\cC_i) \otimes \cC_i^{\omega} \subseteq \Thick \Pic(\cC_i)$ is contained in $\cC_i^{\omega}$, which follows from \cref{superlemma}. Furthermore, note that the composite $\cA^{\omega} \circ (\mathbf{1}_{\CC}\boxtimes \id)$ of the top horizontal functors in \eqref{eq:action} is equivalent to the identity on $\cC^{\omega}$.

By \Cref{cor:picprotomonoidal}, $\cC$ is symmetric monoidal and we denote the (not necessarily compact) unit by $\mathbf{1}_{\cC}$. The second step in the proof is to show that the ind-extension of $\cA^{\omega}$, i.e., the functor
\[
\xymatrix{\cA\colon \CC \boxtimes \cC \simeq \Ind(\CC^{\omega} \square \cC^{\omega}) \ar[r]^-{\Ind(\cA^{\omega})} & \Ind(\cC^{\omega}) \simeq \cC}
\]
is symmetric monoidal. The first equivalence in the equation follows from specializing Remark 4.8.1.8 in \cite{ha} to $\cK = \{\text{finite simplicial sets}\}$ and $\cK' = \{\text{all simplicial sets}\}$. To prove that $\cA$ is symmetric monoidal, first note that $\cA$ has an oidal structure, since $\cA^{\omega}$ does so as the restriction of a symmetric monoidal functor along the oidal inclusion $\iota$. In view of \cite[Corollary 5.4.4.7]{ha}, it remains to check that $\cA$ sends a quasi-unit on $\CC \boxtimes \cC$ to a quasi-unit on $\cC$. But we have already seen that $\cA \circ (\mathbf{1}_{\CC}\boxtimes \id)$ is equivalent to the identity functor on $\cC$, hence $\cA(\mathbf{1}_{\CC} \otimes \mathbf{1}_{\cC}) \simeq \mathbf{1}_{\cC}$.

Finally, we define the desired functor $\cL$ as the following composite
\[
\xymatrix{\cL\colon \CC \simeq \CC \boxtimes \Sp \ar[r]^-{\id_{\CC} \boxtimes \mathbf{1}_{\cC}} & \CC \boxtimes \cC \ar[r]^-{\cA} & \cC.}
\]  
As a composite of symmetric monoidal functors, $\cL$ has the structure of a symmetric monoidal functor as well, so it remains to show that $\cL$ admits a conservative right adjoint. Both $\cI = \Ind(\iota)$ and $\cL$ preserve colimits, so there is a diagram of adjunctions
\[
\xymatrix{\cC \ar@<0.5ex>[r]^-{\cI} & \CC \ar@<0.5ex>[r]^-{\cL} \ar@<0.5ex>[l]^-{\cJ} & \cC \ar@<0.5ex>[l]^-{\cR}}
\]
with the left adjoints displayed on top. By virtue of the commutative diagram
\[
\xymatrix{\cC \ar[r]^-{\id_{\cC}\boxtimes\mathbf{1}_{\cC}} \ar[d]_{\cI} & \cC \boxtimes \cC \ar[r]^-{\otimes} \ar[d]|-{\cI \boxtimes \id_{\cC}} & \cC \ar[d]^{\id_{\cC}} \\
\CC \ar[r]_{\id_{\CC} \boxtimes \mathbf{1}_{\cC}} & \CC \boxtimes \cC \ar[r]_-{\cA} & \cC}
\]
the composite $\cL \circ \cI$ is equivalent to $\id_{\cC}$, hence $\cJ \circ \cR \simeq \id_{\cC}$ as well. It follows that $\cR$ is conservative, which concludes the proof. 
\end{proof}

As in the proof of \Cref{prop:picproto}, the symmetric monoidal structure on a $\Pic$-generated $\infty$-category $\cC$ gives rise to a restricted action map $\cA\colon \Pic(\cC) \boxtimes \cC^{\omega} \to \cC^{\omega}$. By construction of the $\Pic$-cell filtration and \cref{superlemma}, the adjoint of $\cA$ preserves the filtration, i.e., there is a factorization
\[
\xymatrix{\Pic(\cC) \ar[r] \ar@{-->}[rd]_{\Cayley_{\Pic(\cC)}} & \Fun(\cC^{\omega},\cC^{\omega}) \\
& \Fun^{\id_{\N}}(\cC^{\omega},\cC^{\omega}). \ar[u]}
\]
In particular, any invertible object $P \in \Pic(\cC)$ gives rise to an endofunctor $\Cayley_{\Pic(\cC)}(P)$ of $\PicCell_{s}\cC^{\omega}$ for any $s\ge 0$.

\begin{defn}\label{defn:unformsep}
A unital collection of Pic-generated $\infty$-categories $(\cC_i)_{i \in I}$ is said to be uniformly separated if the associated set of restricted Cayley functors $(\Cayley_{\Pic(\cC_i)})$ has the property that there exists $m \ge 0$ independent of $i\in I$ such that $P_i \in \Pic(\cC_i)$ is trivial if and only if 
\[
\xymatrix{\pi_0\Cayley_{\Pic(\cC_i)}(P_i)\colon \pi_0((\PicCell_{m,i}\cC_i^{\omega})^{\circ}) \ar[r] & \pi_0((\PicCell_{m,i}\cC_i^{\omega})^{\circ}})
\]
is the identity map. Here, the superscript $\circ$ indicates the maximal $\infty$-subgroupoid. 
\end{defn}

The purpose of this definition is to provide a condition which implies that the $\Pic$-generated protoproduct generically detects invertible objects in $(\cC_i)_{i \in I}$, in the sense of the following result:

\begin{prop}\label{prop:picdetection}
Let $(\cC_i)_{i \in I}$ be a uniformly separated unital collections of Pic-generated $\infty$-categories, then the functor of $\infty$-groupoids
\[
\xymatrix{\cL_{\Pic}\colon \Prod{\cF}\Pic(\cC_i) \simeq \Pic(\Prod{\cF}^{\Pic}\Ind\Thick\Pic(\cC_i)) \ar[r]^-{\Pic(\cL)} & \Pic\Prod{\cF}^{\Pic}\cC_i}
\]
is injective on $\pi_0$. Note that the first equivalence follows from \cite[Lemma 5.15]{ultra1}. 
\end{prop}
\begin{proof}
Consider the diagonal composite 
\[
\xymatrixcolsep{4pc}
\xymatrix{\pi_0\Prod{\cF}\Pic(\cC_i) \ar[r]^-{\pi_0\cL_{\Pic}} \ar[rrd]_-{\phi_{\Pic}} & \pi_0\Pic\Prod{\cF}^{\Pic}\cC_i \ar[r]^-{\pi_0\Cayley_{\Pic\Prod{\cF}^{\Pic}\cC_i}} & \pi_0\Fun((\Prod{\cF}^{\Pic}\cC_i)^{\omega},(\Prod{\cF}^{\Pic}\cC_i)^{\omega})^{\circ} \ar[d] \\ && \Fun(\pi_0((\Prod{\cF}^{\Pic}\cC_i)^{\omega})^{\circ},\pi_0((\Prod{\cF}^{\Pic}\cC_i)^{\omega})^{\circ}).}
\]
We claim that the monoid homomorphism $\phi_{\Pic}$ is injective. Working in the homotopy category, let $P=[P_i]_{\cF} \in \pi_0\Prod{\cF}\Pic(\cC_i) \cong \Prod{\cF}\pi_0\Pic(\cC_i)$; note that, without loss of generality, we may assume that $P_i \in \pi_0\Pic(\cC_i)$ is non-trivial for all $i \in I$. By assumption, there exists $m$ such that, for any $i \in I$, there exists an object $X_i \in \PicCell_{m,i}\cC_i^{\omega}$ satisfying $\pi_0\Cayley_{\Pic(\cC_i)}(P_i)(X_i) \neq X_i$. Therefore, we may take 
\[
[X_i]_{\cF} \in \Prod{\cF}\PicCell_{m,i}\cC_i \subseteq (\Prod{\cF}^{\Pic}\cC_i)^{\omega}
\]
to detect that $\phi_{\Pic}(P)$ is not the identity functor. 
\end{proof}

We end this discussion of the salient features of the $\Pic$-generated protoproducts by applying the results above to the two key examples in this paper, namely the $\K$-local categories $\cSpnp$ and the completed Franke categories $\cFrnp$. 

\begin{thm}\label{thm:expicgen}
For any ultrafilter $\cF$ on $\bP$, the protoproducts $\Prod{\cF}^{\Pic}\cSpnp$ and $\Prod{\cF}^{\Pic}\cFrnp$ are $\Pic$-generated. Moreover, there are canonical monomorphisms
\[
\xymatrix{\lambda_{\cF}\colon \Prod{\cF}\pi_0\Pic(\cSpnp) \ar[r] & \pi_0\Pic\Prod{\cF}^{\Pic}\cSpnp}
\]
and
\[
\xymatrix{\lambda_{\cF}^{\alg}\colon \Prod{\cF}\pi_0\Pic(\cFrnp) \ar[r] & \pi_0\Pic\Prod{\cF}^{\Pic}\cFrnp.}
\]
\end{thm}
\begin{proof}
We verify the claim for $\cSpnp$, the one for $\cFrnp$ being proven similarly. Since the unit $L_{\E}S^0 \in \Spnp$ is compact, it follows from  \Cref{prop:localdualitycontext}(3) and \Cref{cor:savetheday} that $\cSpnp$ is $\Pic$-generated. In light of \Cref{prop:picproto} and \Cref{prop:picdetection}, it remains to show that $(\cSpnp)_{p \in \bP}$ is a unital and uniformly separated collection. To this end, recall the generalized Moore spectra $M^I(n)$ from \Cref{ssec:ex} that may be built using $2^n$ $\Pic$-cells, i.e., $M^I(n) \in \PicCell_{2^n,p}\cSpnp$ for all $p \in \bP$. Applying $\K$-local Spanier Whitehead duality $D$ to \eqref{eq:completionformula} and specializing to $X = S^0$ yields 
\[
L_{\K}S^0 \simeq \colim_I DM_I(n) \in \Ind\PicCell_{2^n,p}\cSpnp,
\]
so the collection $(\cSpnp)_{p \in \bP}$ is unital. 

In the proof of \cite[Proposition 14.3]{hoveystrickland}, Hovey and Strickland show that for any non-trivial $P \in \Pic(\cSpnp)$ there exists $I$ with $P \otimes M^I(n)$ not equivalent to $M^I(n)$. Thus $(\cSpnp)_{p \in \bP}$ is uniformly separated with $m=2^n$. 

The analogous arguments prove that the collection $(\cFrnp)_{p \in \bP}$ is unital and uniformly separated.
\end{proof}

\subsection{Torsion in protoproducts}\label{ssec:prototorsion}

Let $(R_i)_{i \in I}$ be a collection of even periodic $\bE_{\infty}$-ring spectra indexed on a set $I$. Let $(x_{1,i},\ldots,x_{n,i})$ be a regular sequence in $\pi_0 R_i$. For $k_i \in \N$, define the compact $R_i$-module $\kappa_{i}^{(k_i)}$ to be $R_i/(x_{1,i}^{k_i},\ldots,x_{n,i}^{k_i})$ so that $\pi_*\kappa_{i}^{(k_i)}$ is an even periodic $\pi_*R_i$-module with $\pi_0 \kappa_{i}^{(k_i)} \cong (\pi_0R_i)/(x_{1,i}^{k_i},\ldots,x_{n,i}^{k_i})$. For each $i\in I$, let $A_i \in \CAlg_{R_i}$ be a flat $R_i$-module (i.e., $\pi_*A_i$ is a flat $\pi_*R_i$-module) and define $M^{(k_i)}(A_i) = A_i \otimes_{R_i}\kappa_{i}^{(k_i)}$. Note that $M^{(k_i)}(A_i)$ is a compact $A_i$-module built out of $2^n$ $A_i$-cells. Let $J_i = (x_{1,i},\ldots,x_{n,i})$, the category of $J_i$-torsion $A_i$-modules is the localizing subcategory 
\[
\tMod_{A_i} = \Loc(M^{(1)}(A_i)) = \Loc(\{M^{(k_i)}(A_i)| k_i \in \N\})
\]
of $\Mod_{A_i}$ generated by $M^{(1)}(A_i)$. We will write 
\[
\iota_{A_i} \colon \tMod_{A_i} \to \Mod_{A_i}
\] 
for the corresponding inclusion functor and $\Gamma_{A_i}$ for its right adjoint. Since $\Mod_{A_i}$ is compactly generated by its unit $A_i$, it follows that $\cG_{i}^{\tors} = \Gamma_{A_i}A_i$ is an invertible generator for $\tMod_{A_i}$.

\begin{ex} 
The example to keep in mind is $I = \bP$ and, for $p$ a prime, $R_p = \E$, $\kappa_{p}^{(k_p)} \simeq \E/(p^{k_p}, \ldots, u_{n-1}^{k_p})$, and $A_p = \E^{\otimes s}$ for some $s \ge 1$. The category $\tMod_{\E^{\otimes s}}$ then coincides with the category of $I_n$-torsion $\E^{\otimes s}$-modules. Similarly, we may consider the algebraic analogues $R_p = (\E)_{\star}$, $\kappa_{p}^{(1)} = (\E)_{\star}/I_n$, and $A_p = (\E^{\otimes s})_{\star}$.  
\end{ex}
 Suppose $\cF$ is an ultrafilter on $I$, we define:
\begin{defn}\label{defn:ultratorsion}
An object $W \in \Prod{\cF}^{\flat} \Mod_{A_i}$ is said to be  $\pi_{[*]}$-torsion if every element in $\pi_{[*]}W$ is annihilated by $[J_i]_{\cF}^{[d_i]} = [J_i^{d_i}]_{\cF}$ for some $[d_i] \in \N^{\cF}$.
\end{defn}

For example, for every  $[k_i] \in \N^{\cF}$, $M^{[k_i]}:= [M^{(k_i)}(A_i)]_{\cF}$ is a compact $\pi_{[*]}$-torsion object in $\Prod{\cF}^{\flat} \Mod_{A_i}$. The goal of this section is to show that all torsion objects can be built from the set of compact objects $\{M^{[k_i]}|[k_i] \in \N^{\cF}\}$.

\begin{lem}\label{lem:torsioncover}
If $W  \in \Prod{\cF}^{\flat} \Mod_{A_i}$ is $\pi_{[*]}$-torsion, then there exists an object \[M \in \Loc(\{M^{[k_i]} | [k_i] \in \N^{\cF}\}) \subseteq  \Prod{\cF}^{\flat} \Mod_{A_i}\] together with a map $f\colon M \to W$ which is surjective on homotopy groups. 
\end{lem}
\begin{proof}
For any $w= [w_i] \in \pi_{[*]}W$, there exists  $[k_i]_w\in \N^{\cF}$ and a map $f_{w}\colon M^{[k_i]_w} \to W$ with $w \in \im(\pi_{[*]}f_{w})$. Explicitly, we may choose $(k_i)_{i\in I} \in \N^I$ so that
\[
\{i\colon J_i^{k_i}w_i = 0\} \in \cF.\]
Consequently, the map
\[
\xymatrix{f=\bigoplus f_w\colon M= \bigoplus \limits	_{w \in \pi_{[*]}W} M^{[k_i]_w} \ar[r] & W} 
\]
gives what we want. 
\end{proof}

\begin{lem}\label{lem:torsionresolution}
Suppose $W \in \Prod{\cF}^{\flat}\Mod_{A_i}$ and we are given a map $f\colon M \to W$ with $M \in \Loc(\{M^{[k_i]} | [k_i] \in \N^{\cF}\})$ and $\pi_{[*]}f$ surjective, then there exists $N \in \Loc(\{M^{[k_i]} | [k_i] \in \N^{\cF}\})$ and a factorization
\[
\xymatrix{M \ar[r]^-{j} \ar[d]_{f} & N \ar[ld]^{g}\\
W}
\]
such that $\pi_{[*]}g$ is surjective and $\ker(\pi_{[*]}f) \subseteq \ker(\pi_{[*]}j)$.
\end{lem}
\begin{proof}
Let $F$ be the fiber of $f\colon M \to W$, so that $\pi_{[*]}F \cong \ker(\pi_{[*]}f)$. Since $M$ is $\pi_{[*]}$- torsion, so is $\pi_{[*]}F$ and we can apply \Cref{lem:torsioncover} to obtain a map $h\colon N' \to F$ with $N' \in \Prod{\cF}^{\flat}\tMod_{A_i}$ and $\pi_{[*]}h$ surjective. Consider the following commutative diagram of cofiber sequences:
\[
\xymatrix{& F \ar[d] & \\
N' \ar[r] \ar[ru]^h & M \ar[r]^j \ar[d]_f & N \ar@{-->}[ld]^g \\
& W.}
\]
The dashed map $g$ exists because the composite $N' \to M \to W$ factors through $F$ and is thus null. Since $\pi_{[*]}f$ is surjective, so is $\pi_{[*]}g$. To verify that $\ker(\pi_{[*]}f) \subseteq \ker(\pi_{[*]}j)$, observe that, by construction, any $m \in \pi_{[*]}M$ with $\pi_{[*]}f(m)=0$ lifts to an element $m' \in \pi_{[*]}N'$, hence $\pi_{[*]}j(m) = 0$. 
\end{proof}

For each $i \in I$, recall that $\cG_{i}^{\tors} = \{ \Sigma^k\Gamma_{A_i}A_i|k \in \Z\}$ is a set of Pic-generators for $\tMod_{A_i}$. The resulting $\cG_{i}^{\tors}$-cell filtration on $\tMod_{A_i}$ will be denoted by $\cG_{i}^{\tors}\Cell_*((\tMod_{A_i})^{\omega})$, see \Cref{ex:cellproto}.

\begin{defn}
With notation as above, suppose $\cF$ is an ultrafilter on $I$, then we define the protoproduct of the torsion categories $\tMod_{A_i}$ as
\[
\Prod{\cF}^{\flat} \tMod_{A_i} = \Prod{\cF}^{\Pic}(\tMod_{A_i},\cG_{i}^{\tors}) = \Prod{\cF}^{\flat} (\tMod_{A_i}, \cG_{i}^{\tors}\Cell_*((\tMod_{A_i})^{\omega})).
\]
\end{defn}

Since the unit $\Gamma_{A_i}A_i$ of $\tMod_{A_i}$ is contained in $\Ind\cG_{i}^{\tors}\Cell_{2^n}((\tMod_{A_i})^{\omega})$, the protoproduct $\Prod{\cF}^{\flat} \tMod_{A_i}$ comes equipped with a natural symmetric monoidal structure by \Cref{cor:picprotomonoidal}.

\begin{lem}\label{lem:torsioniota}
There is a fully faithful symmetric oidal functor
\[
\xymatrix{\iota_{\cF}=(\iota_{A_i})_{\cF} \colon \Prod{\cF}^{\flat} \tMod_{A_i} \ar[r] & \Prod{\cF}^{\flat} \Mod_{A_i}.}
\]
\end{lem}
\begin{proof}
Let $i \in I$. By the construction of $\Gamma_{A_i}$, there exists a cofinal sequence of elements $(k_i) \in \N^n$ and a natural equivalence 
\[
\colim_{(k_i)}M^{(k_i)}(A_i) \simeq \Gamma_{A_i}A_i = \cG_i^{\tors},
\]
so that $\cG_i^{\tors} \in \Ind\Cell_{2^n}\Mod_{A_i}^{\omega}$. If $X \in \cG_{i}^{\tors}\Cell_k((\tMod_{A_i})^{\omega})$ for some $k \in \N$, then $X \otimes \cG_i^{\tors} \simeq X$, so the compactness of $X \in (\tMod_{A_i})^{\omega}$ provides the existence of a sequence $(k_i)$ such that $X$ is a retract of $X \otimes M^{(k_i)}(A_i)$. This implies that $X \in \Cell_{2^nk}\Mod_{A_i}$ and hence
\[
\cG_{i}^{\tors}\Cell_k((\tMod_{A_i})^{\omega}) \subseteq \Cell_{2^nk}\Mod_{A_i}.
\]
By passing to the associated protoproducts, we thus obtain from \Cref{cor:symmonfun} a fully faithful symmetric oidal functor $\iota_{\cF}$, as claimed.
\end{proof}

Note that $\iota_{\cF}$ preserves colimits and compact objects and that, as mentioned before, $M^{[k_i]}$ is a compact object in $\Prod{\cF}^{\flat} \tMod_{A_i}$ for all $[k_i] \in \N^{\cF}$. We are now ready to state the main result of this section.

\begin{prop}\label{prop:protopitorsion}
Let $W \in \Prod{\cF}^{\flat} \Mod_{A_i}$. The following are equivalent:
\begin{enumerate}
\item $W$ is in $\Loc(\{M^{[k_i]} | [k_i] \in \N^{\cF}\})$.
\item $W$ is in the essential  image of $\iota_{\cF}$.
\item $W$ is $\pi_{[*]}$-torsion.
\end{enumerate} 
\end{prop}
\begin{proof}
For every $[k_i] \in \N^{\cF}$ the object $M^{[k_i]} $ is in  the essential  image of $\iota_{\cF}$  so (1) implies (2).  It is clear that (2) implies (3). We will show that (3) implies (1).
 Assume $W \in \Prod{\cF}^{\flat} \Mod_{A_i}$ is $\pi_{[*]}$-torsion. We will construct an Adams type resolution 
\[
\xymatrix{M_0 \ar[r]^-{j_0} \ar[d]_{f_0} & M_1 \ar[r]^-{j_1} \ar[dl]_-{f_1} & M_2 \ar[r]^-{j_2} \ar[dll]^-{f_2} & \ldots \ar@{}[dlll]^-{\ldots} \\
W}
\]
to show that  $W$ is in $\Loc(\{M^{[k_i]} | [k_i] \in \N^{\cF}\})$. Indeed, using \Cref{lem:torsioncover} and \Cref{lem:torsionresolution} iteratively implies that there is such a resolution satisfying the following: For all $k\ge 0$
	\begin{enumerate}
		\item $M_k \in \Loc(\{M^{[k_i]} | [k_i] \in \N^{\cF}\})$,
		\item $\pi_{[*]}f_k$ is surjective, and
		\item $\ker(\pi_{[*]}f_k) \subseteq \ker(\pi_{[*]}j_k)$.
	\end{enumerate}
Thus the structure maps in the induced diagram
\[
\xymatrix{\fib(f_0) \ar[r] & \fib(f_1) \ar[r] & \fib(f_2) \ar[r] & \ldots} 
\]
are trivial in homotopy, hence \cite[Proposition 3.59]{ultra1} implies that $\colim \fib(f_k) \simeq 0$ and so we see that $\colim M_k \simeq W$ as desired. 
\end{proof}

\begin{defn}
With notation as above, define $(\Prod{\cF}^{\flat}\Mod_{A_i})^{\tors}$ as the localizing subcategory in $\Prod{\cF}^{\flat}\Mod_{A_i}$ generated by  $\{ [\Sigma^{n_{i}}M^{(k_i)}(A_i)]_{\cF}| [n_{i}]\in \Z^{\cF}, [k_i] \in \N^{\cF}\}$. 
\end{defn}

\begin{cor}\label{cor:prototorsion}
For any ultrafilter $\cF$ on $I$, there is a canonical symmetric monoidal equivalence
\[
\xymatrix{\Prod{\cF}^{\flat}(\tMod_{A_i}) \ar[r]^-{\sim} & (\Prod{\cF}^{\flat}\Mod_{A_i})^{\tors}.}
\]
\end{cor}
\begin{proof}
\Cref{prop:protopitorsion} identifies the essential image of $\iota_{\cF}$ with $(\Prod{\cF}^{\flat}\Mod_{A_i})^{\tors}$, so it remains to verify that the induced symmetric oidal equivalence is unital. This is a formal consequence of the fact that, in light of \Cref{cor:picprotomonoidal}, both $\infty$-categories are unital.
\end{proof}

\section{The proof of the main theorem}

In this section, we will put the pieces together to prove \Cref{thm:knmain}. The missing ingredients at this point are the relationship between descent and the Pic-generated protoproducts of \cref{ssec:four} as well as the compatibility of the main equivalence of \cite{ultra1} with passing to torsion subcategories.

Throughout this section, let $I$ be a set and let $(\cC_i, F_i)_{i\in I}$ be a collection of local duality contexts such that $\cC_i$ is Pic-generated for each $i \in I$. Let $(\ccC_i,\cotimes)_{i\in I}$ be the collection of presentably symmetric monoidal stable $\infty$-categories of $F_i$-complete objects in $\cC_i$ as in \cref{Section3.1}. Recall from \cref{prop:localdualitycontext} and \cref{cor:savetheday}, that $\ccC_i$ is compactly generated and Pic-generated for each $i \in I$. We will assume that the collection $(\cC_i)_{i \in I}$ is unital in the sense of \Cref{def:uniformunital}. We remind the reader that the unit object in $\ccC_i$ is not assumed to be compact. 

Given $A_i \in \CAlg(\cC_i)$, we write $\cMod_{A_i}$ for the $\infty$-category of $L_{F_i}A_i$-modules in $\ccC_i$.

\subsection{Reduction to the totalization: Descent}

We refer to \cite[Section 5.1]{ultra1} for some background material on descent and, in particular, the terminology we are going to use throughout this subsection. 

\begin{defn}
A collection of commutative algebras $(A_i)_{i\in I}$ is said to be uniformly descendable in $(\ccC_i)_{i \in I}$ if there exists an integer $r>0$ such that for all $i\in I$, $A_i$ is descendable of fast-degree less than or equal to $r$. If additionally the $A_i$-based Adams spectral sequence for $\End(1_{\ccC_i})$ collapses at the $E_2$-page for almost all $i \in I$, then we call $(A_i)_{i\in I}$ strongly uniformly descendable. 
\end{defn}

The next result generalizes \cite[Theorem 5.1]{ultra1} to the modified Pic-generated protoproduct of a collection $(\ccC_i,\cotimes)_{i\in I}$ in which the corresponding units are not necessarily compact. As explained in \cite[Section 4.2]{ultra1}, the Amitsur complex associated to $A_i$ provides a diagram of cosimplicial symmetric monoidal $\infty$-categories $\cMod_{A_i^{\otimes \bullet +1}}$. 

\begin{prop}
Let $\cF$ be an ultrafilter on $I$ and suppose that $(A_i)_{i\in I}$ is strongly uniformly descendable, then there is a canonical symmetric monoidal equivalence
\[
\xymatrix{\Prod{\cF}^{\Pic}\ccC_i \ar[r]^-{\simeq} & \Loc\Pic\Tot\Prod{\cF}^{\Pic}\cMod_{A_i^{\otimes \bullet +1}}.}
\]
\end{prop}
\begin{proof}
By \Cref{cor:picprotofun} and for any $s\ge 0$, the canonical symmetric monoidal and colimit preserving functors $\ccC_i \to \cMod_{A_i^{\otimes s+1}}$ induce a symmetric monoidal functor after applying the Pic-generated protoproduct. This provides a canonical symmetric monoidal and colimit preserving functor
\[
\xymatrix{\Xi\colon \Prod{\cF}^{\Pic}\ccC_i \ar[r] & \Tot\Prod{\cF}^{\Pic}\cMod_{A_i^{\otimes \bullet +1}}.}
\]
We claim that $\Xi$ is fully faithful and induces an equivalence on Picard spectra: The proof of \cite[Proposition 5.13]{ultra1} can be adapted easily to give fully-faithfulness of $\Xi$, while the same argument as for \cite[Proposition 5.16]{ultra1} shows the claim about Picard spectra. 

It follows from \Cref{prop:picproto} that $\Xi$ descends to a canonical symmetric monoidal equivalence 
\[
\xymatrix{\Prod{\cF}^{\Pic}\ccC_i \ar[r]^-{\simeq} & \Loc\Pic\Tot\Prod{\cF}^{\Pic}\cMod_{A_i^{\otimes \bullet +1}}.}
\]
\end{proof}

We next verify that the conditions of the previous proposition are satisfied for the two main examples in this paper. This is mostly a matter of collecting results from the literature. Recall that $\bP$ is the set of prime numbers.

\begin{prop}
Let $n\ge 0$ be an integer.
\begin{enumerate}
	\item If $p$ is an odd prime with $2p-2 \ge n^2$, then $(\E)_{p\in \bP}$ is strongly uniformly descendable in $(\cSpnp)_{p\in \bP}$.  
	\item If $p$ is a prime with $p>n+1$, then $(\A)_{p \in \bP}$ is strongly uniformly descendable in $(\cFrnp)_{p\in \bP}$. 
\end{enumerate}
\end{prop}
\begin{proof}
If $p>n+1$, the Morava stabilizer group has no $p$-torsion, which implies that it has finite cohomological dimension, see \cite[Proposition 2.2.2]{moravastack}. As in the proof of \cite[Lemma 5.33]{ultra1}, this shows that $(\A)_{p \in \bP}$ is strongly uniformly descendable in $(\cFrnp)_{p \in \bP}$. Lifting the finite cohomological dimension of the stabilizer group to $\Spnp$ as in the proof of the smash product theorem \cite[Section 8]{ravbook2} yields that $\E$ is descendable of uniformly bounded fast degree, see also \cite[Proposition 10.10]{mathew_galois}. Furthermore, the sparsity argument in \cite[Proposition 7.5]{picard} provides the collapsing of the $\K$-local $\E$-based Adams spectral sequence for $L_{\K}S^0$ with uniformly bounded intercept in the given range, hence  $(\E)_{p\in \bP}$ is uniformly descendable in $(\cSpnp)_{p \in \bP}$. 
\end{proof}

By the proof \Cref{thm:expicgen}, both $(\cSpnp)_{p \in \bP}$ and $(\cFrnp)_{p \in \bP}$ are unital collections of Pic-generated $\infty$-categories. The following corollary is now immediate:

\begin{cor}\label{cor:totalization}
For any ultrafilter on the set of primes $\bP$ there are canonical symmetric monoidal equivalences
\[
\xymatrix{\Prod{\cF}^{\Pic}\cSpnp \ar[r]^-{\simeq} & \Loc\Pic\Tot\Prod{\cF}^{\Pic}\cMod_{\E^{\otimes \bullet +1}}, \quad \Prod{\cF}^{\Pic}\cFrnp \ar[r]^-{\simeq} & \Loc\Pic\Tot\Prod{\cF}^{\Pic}\cMod_{\A^{\otimes \bullet +1}}.}
\]
\end{cor}

\subsection{Reduction to the cell-protoproduct: Picard groups}

The next step in the proof is to replace the modified Pic-generated protoproduct by the cell-protoproduct, so that we can apply the cosimplicial formality theorem of \cite[Section 4]{ultra1}. To this end, let $\cG_{A_i^{\otimes s}}^{\comp}$ be the unit in $\cMod_{A_i^{\otimes s}}$ and consider the cell-protoproduct 
\[
\Prod{\cF}^{\flat}\cMod_{A_i^{\otimes s}} 
= \Prod{\cF}^{\Pic}(\cMod_{A_i^{\otimes s}}, \cG_{A_i^{\otimes s}}^{\comp}) 
= \Prod{\cF}^{\flat}\left (\cMod_{A_i^{\otimes s}}, \cG_{A_i^{\otimes s}}^{\comp}\Cell_*(\cMod_{A_i^{\otimes s}}^{\omega}) \right )
\]
as in \Cref{ex:cellproto}.

\begin{prop}\label{prop:pictoflatproto}
Let $\cF$ be an ultrafilter on $I$ and suppose that $(A_i)_{i\in I}$ is strongly uniformly descendable. If the Picard groups $\pi_0\Pic(\cMod_{A_i})$ are generated by the corresponding suspension functor for each $i \in I$, then there is a canonical symmetric monoidal equivalence
\[
\xymatrix{\Tot\Prod{\cF}^{\flat}\cMod_{A_i^{\otimes \bullet +1}} \ar[r]^-{\simeq} & \Tot\Prod{\cF}^{\Pic}\cMod_{A_i^{\otimes \bullet +1}}.}
\]
\end{prop}
\begin{proof}
The assumption on the Picard groups guarantees that the Pic-cell filtration coincides with the cell filtration, hence we obtain a canonical symmetric monoidal equivalence
\[
\xymatrix{\Prod{\cF}^{\flat}\cMod_{A_i} \ar[r]^-{\simeq} & \Prod{\cF}^{\Pic}\cMod_{A_i}.}
\]
The comparison lemma \cite[Lemma 5.19]{ultra1} thus reduces the proof to showing that the canonical maps
\[
\xymatrix{\Prod{\cF}^{\flat}\cMod_{A_i^{\otimes s+1}} \ar[r] & \Prod{\cF}^{\Pic}\cMod_{A_i^{\otimes s+1}}}
\]
are fully faithful for all $s\ge 0$. This is a consequence of the construction, as the constituent maps between the filtration steps are fully faithful, thereby finishing the proof.
\end{proof}

\begin{cor}\label{cor:flatproto}
For any ultrafilter on $\bP$ there are canonical symmetric monoidal equivalences
\[
\xymatrix{\Tot\Prod{\cF}^{\flat}\cMod_{\E^{\otimes \bullet +1}} \ar[r]^-{\simeq} & \Tot\Prod{\cF}^{\Pic}\cMod_{\E^{\otimes \bullet +1}}, \quad \Tot\Prod{\cF}^{\flat}\cMod_{\A^{\otimes \bullet +1}} \ar[r]^-{\simeq} & \Tot\Prod{\cF}^{\Pic}\cMod_{\A^{\otimes \bullet +1}}.}
\]
\end{cor}
\begin{proof}
By virtue of \Cref{prop:etheory} and \Cref{cor:frmodpicard}, these equivalences are consequences of \Cref{prop:pictoflatproto}.
\end{proof}

\subsection{Finishing the proof: Formality} \label{sec:end}
We will focus on the two main examples of this paper, which are discussed at a finite prime in detail in \Cref{ssec:ex}, so let $I = \bP$. On the topological side, we take $(\ccC_i,A_i)_{i\in I}$ to be $(\cSpnp,\E)_{p \in \bP}$ and define $\topkappa^{(k_p)} = \E/(p^{k_p},\ldots,u_{n-1}^{k_p})$ so that 
\[
\Mtop^{(k_p)}(\E^{\otimes s+1}) = \E^{\otimes s+1} \otimes_{\E} \topkappa^{(k_p)} \simeq \E^{\otimes s+1}/(p^{k_p},\ldots,u_{n-1}^{k_p})
\]
for all $s\ge 0$ and $k_p \ge 1$. Analogously, on the algebraic side we consider $(\cFrnp,\A)_{p \in \bP}$. For $k_p \ge 1$ an integer, we set $\algkappa^{(k_p)} = (\E)_{\star}/(p^{k_p},\ldots,u_{n-1}^{k_p})_{\star}$ and, for $s\ge 0$, $\Malg^{(k_p)}(\A^{\otimes s+1}) = (\E^{\otimes s+1})_{\star}/(p^{k_p},\ldots,u_{n-1}^{k_p})_{\star}$. This choice of notation is justified by the Morita equivalence of \eqref{eq:frmorita}. The notion of torsion in the rest of this section will consequently be determined by the ideals $(p^{k_p},\ldots,u_{n-1}^{k_p})$ and $(p^{k_p},\ldots,u_{n-1}^{k_p})_{\star}$ as in \Cref{ssec:prototorsion}. In particular:

\begin{defn}
For all $s\ge 0$ and any ultrafilter $\cF$ on $\bP$, we define localizing subcategories
\[
(\Prod{\cF}^{\flat}\Mod_{\E^{\otimes s +1}})^{\tors} = \Loc([\Mtop^{(k_p)}(\E^{\otimes s+1})]_{\cF}|[k_p]_{\cF} \in \N^{\cF}) \subseteq \Prod{\cF}^{\flat}\Mod_{\E^{\otimes s +1}}
\]
and
\[
(\Prod{\cF}^{\flat}\Mod_{(\E^{\otimes s +1})_{\star}})^{\tors}  = \Loc([\Malg^{(k_p)}(\A^{\otimes s+1})]_{\cF}|[k_p]_{\cF} \in \N^{\cF}) \subseteq \Prod{\cF}^{\flat}\Mod_{(\E^{\otimes s +1})_{\star}.}
\]
\end{defn}

Recall from \cite[Theorem 5.38]{ultra1} that we have established a canonical symmetric monoidal equivalence 
\[
\xymatrix{\Phi_{\bullet+1}\colon \Prod{\cF}^{\flat}\Mod_{\E^{\otimes \bullet +1}} \ar[r]^-{\simeq} & \Prod{\cF}^{\flat}\Mod_{(\E^{\otimes \bullet +1})_{\star}}}
\]
of cosimplicial compactly generated $\Q$-linear stable $\infty$-categories. The next result establishes the compatibility of this equivalence with the torsion objects. 

\begin{lem}\label{lem:kappaformality}
For all $s\ge 0$ and $k_p \ge 1$ and any non-principal ultrafilter $\cF$ on $\bP$, the functor $\Phi_{s+1}$ sends $[\Mtop^{(k_p)}(\E^{\otimes s+1})]_{\cF}$ to $[\Malg^{(k_p)}(\A^{\otimes s+1})]_{\cF}$. 
\end{lem}
\begin{proof}
By construction, it suffices to check the statement for $s=0$, i.e., that $\Phi_1([\topkappa^{(k_p)}]_{\cF}) \simeq [\algkappa^{(k_p)}]_{\cF}$ for all $(k_p \ge 1)_{p \in \bP}$. The equivalence $\Phi_{1}$ is symmetric monoidal by \cite[Theorem 5.38]{ultra1}, so in particular preserves the tensor unit. Moreover, by \cite[Proposition 6.7]{ultra1}, the natural map
\[
\xymatrix{\pi_0[\E]_{\cF} \cong \pi_0\End([\E]_{\cF}) \ar[r]_-{\cong}^-{\pi_0\Phi_1} & \pi_0\End([(\E)_{\star}]_{\cF}) \cong \pi_0[(\E)_{\star}]_{\cF}}
\]
induced by $\Phi_1$ can be identified with the isomorphism 
\[
\xymatrix{\pi_0[\E]_{\cF} \cong \End(\pi_0[\E]_{\cF}) \ar[r]_-{\cong}^-{\Phi_1'\pi_0} & \End(\pi_0[(\E)_{\star}]_{\cF}) \cong \pi_0[(\E)_{\star}]_{\cF}}
\]
induced by $\varphi = \Prod{\cF}\varphi_p \colon \Prod{\cF}\pi_0\E \cong \Prod{\cF}\pi_0(\E)_{\star}$.  Therefore $\pi_0\Phi_1$ sends the sequence $([p^{k_p}]_{\cF},\ldots,[u_{n-1}^{k_p}]_{\cF})$ to the sequence $([p^{k_p}]_{\cF},\ldots,[u_{n-1}^{k_p}]_{\cF})$; the claim follows. 
\end{proof}

\begin{prop}\label{prop:torsionformality}
For any non-principal ultrafilter $\cF$ on $\bP$, there is a symmetric monoidal equivalence of cosimplicial compactly generated $\Q$-linear stable $\infty$-categories
\[
\xymatrix{\Tot(\Prod{\cF}^{\flat}\tMod_{(\E^{\otimes \bullet+1})}) \ar[r]^-{\simeq} & \Tot(\Prod{\cF}^{\flat} \tMod_{(\E^{\otimes \bullet+1})_{\star}}).}
\]
\end{prop}
\begin{proof}
For any $s\ge 0$, consider the following commutative diagram, in which all functors are symmetric monoidal and colimit preserving:
\[
\xymatrix{
\Prod{\cF}^{\flat}\tMod_{\E^{\otimes s +1}} \ar[r]^-{\simeq} \ar@{-->}[d]_-{\simeq}^-{\Phi_{s +1}^{\tors}} &
(\Prod{\cF}^{\flat}\Mod_{\E^{\otimes s +1}})^{\tors} \ar[r] \ar[d]_{\simeq}^-{\Phi_{s +1}} &
\Prod{\cF}^{\flat}\Mod_{\E^{\otimes s +1}} \ar[d]_{\simeq}^-{\Phi_{s +1}} \\
\Prod{\cF}^{\flat}\tMod_{(\E^{\otimes s +1})_{\star}} \ar[r]^-{\simeq} &
(\Prod{\cF}^{\flat}\Mod_{(\E^{\otimes s +1})_{\star}})^{\tors} \ar[r] &
\Prod{\cF}^{\flat}\Mod_{(\E^{\otimes s +1})_{\star}}.}
\]
The right vertical functor is an equivalence by \cite[Theorem 5.38]{ultra1} and hence restricts to an equivalence on localizing ideals generated by $[\Mtop^{(k_p)}(\E^{\otimes s+1})]_{\cF}$ and $[\Malg^{(k_p)}(\A^{\otimes s+1})]_{\cF}$ for all $[k_p] \in \N^{\cF}$, respectively, using \Cref{lem:kappaformality}. \Cref{cor:prototorsion} yields the horizontal equivalences in the left square. Since the generators of the torsion categories are induced up from the zeroth level of the respective cosimplicial diagram, the equivalences $\Phi_{\bullet +1}^{\tors}$ are compatible with the cosimplicial structure maps. Therefore, we obtain a symmetric monoidal equivalence upon passage to totalizations.
\end{proof}

We are now ready for the proof of the main theorem.

\begin{proof}[Proof of \Cref{thm:knmain}]
There is a string of symmetric monoidal equivalences
\begin{align*}
\Tot\Prod{\cF}^{\Pic}\cMod_{(\E^{\otimes \bullet+1})} & \simeq \Tot\Prod{\cF}^{\flat}\cMod_{\E^{\otimes \bullet+1}} && \text{by \Cref{cor:flatproto}} \\
& \simeq \Tot\Prod{\cF}^{\flat}\tMod_{\E^{\otimes \bullet+1}} && \text{by \Cref{prop:etheory}} \\
& \simeq \Tot\Prod{\cF}^{\flat}\tMod_{(\E^{\otimes \bullet+1})_{\star}} && \text{by \Cref{prop:torsionformality}} \\
& \simeq \Tot\Prod{\cF}^{\flat}\cMod_{A_{n,p}^{\otimes \bullet +1}} && \text{by \Cref{prop:frtorsionmorita}} \\ 
& \simeq \Tot\Prod{\cF}^{\Pic}\cMod_{A_{n,p}^{\otimes \bullet +1}} && \text{by \Cref{cor:flatproto}}.
\end{align*}
Note that this equivalence relies on the main result of \cite{ultra1} via \cref{prop:torsionformality}. It thus follows from \Cref{cor:totalization} that there are symmetric monoidal equivalences
\[
\Prod{\cF}^{\Pic} \cSpnp \simeq \Loc\Pic\Tot\Prod{\cF}^{\Pic}\cMod_{(\E^{\otimes \bullet+1})} \simeq \Loc\Pic\Tot\Prod{\cF}^{\Pic}\cMod_{A_{n,p}^{\otimes \bullet +1}} \simeq \Prod{\cF}^{\Pic} \cFrnp
\]
which finishes the proof. 
\end{proof}

\subsection{Comparison to the $E$-local category}\label{ssec:comparison}

The goal of this section is to compare the Pic-generated protoproduct of the $\K$-local categories to the Pic-generated protoproduct of the $\E$-local category $\Spnp$ studied in \cite{ultra1}. As a consequence, we will establish the compatibility of the main equivalence of \cite{ultra1} with the equivalence of \Cref{thm:knmain}.

We begin with the construction of a topological comparison functor, leaving the (entirely notational) modifications necessary on the algebraic side to the interested reader. Recall the symmetric monoidal local duality equivalence 
\[
\xymatrix{\M \colon \cSpnp \ar@<0.5ex>[r] &  \Mnp \noloc  L_{\K} \ar@<0.5ex>[l],}
\]
between the monochromatic category $\Mnp$ and the $\K$-local category $\cSpnp$ from \eqref{eq:toplocalduality}. For any $s\ge 0$, there are commutative squares
\begin{equation}\label{eq:comparisonsquare}
\xymatrixcolsep{3pc}
\xymatrix{\Spnp \ar[r]^-{-\otimes \E^{\otimes s+1}} \ar[d]_{\M} & \Mod_{\E^{\otimes s+1}} \ar[d]^{\Gamma_{\E^{\otimes s+1}}} \\
\Mnp \ar[r]_-{-\otimes \E^{\otimes s+1}} & \tMod_{\E^{\otimes s+1}},}
\end{equation}
where $\Gamma_{\E^{\otimes s+1}}$ is the functor constructed in \Cref{ssec:prototorsion}. These squares are compatible with each other for varying $s$. 

For all $s\ge 0$, the functors appearing in the square \eqref{eq:comparisonsquare} are symmetric monoidal and preserve colimits. Furthermore, \eqref{eq:completionformula} shows that, for any $P \in \Pic(\Spnp)$, the object $\M(P)$ is a filtered colimit of spectra $P \otimes M_p^I(n) \in (\Mnp)^{\omega}$ that can be built from elements in $\Pic(\Mnp)$ in $2^n$ steps. In other words, if we equip each of the categories with their Pic-cell filtrations, then \eqref{eq:comparisonsquare} restricts to a commutative square 
\[
\xymatrixcolsep{3pc}
\xymatrix{\PicCell_k(\Spnp^{\omega}) \ar[r]^-{-\otimes \E^{\otimes s+1}} \ar[d]_{\M} & \PicCell_{k}(\Mod_{\E^{\otimes s+1}}^{\omega}) \ar[d]^{\Gamma_{\E^{\otimes s+1}}} \\
\Ind \PicCell_{k \cdot 2^n}((\Mnp)^{\omega}) \ar[r]_-{-\otimes \E^{\otimes s+1}} & \Ind \PicCell_{k \cdot 2^n}((\tMod_{\E^{\otimes s+1}})^{\omega})}
\]
for any $s\ge 0$ and any $k \ge 0$. Note that the analogous claim for the $\Ind$-completed analogue of the right vertical arrow is the content of the proof of \Cref{lem:torsioniota}. This puts us in the situation of \Cref{prop:superfunctors}. By taking the totalization of the resulting squares, we obtain the topological part of the following proposition; we omit the analogous details for the algebraic one:

\begin{prop}\label{prop:torscomparisonsquare}
For any ultrafilter $\cF$ on $\cP$, there are commutative squares
\[
\xymatrix{\Prod{\cF}^{\Pic}\Spnp \ar[r] \ar[d] & \Tot\Prod{\cF}^{\flat}\Mod_{\E^{\otimes \bullet +1}} \ar[d] & \Prod{\cF}^{\Pic}\Frnp \ar[r] \ar[d] & \Tot\Prod{\cF}^{\flat}\Mod_{(\E^{\otimes \bullet +1})_{\star}} \ar[d] \\
\Prod{\cF}^{\Pic}\Mnp \ar[r] & \Tot\Prod{\cF}^{\flat}\tMod_{\E^{\otimes \bullet +1}} & \Prod{\cF}^{\Pic}\tFrnp \ar[r] & \Tot\Prod{\cF}^{\flat}\tMod_{(\E^{\otimes \bullet +1})_{\star}},}
\]
in which the horizontal functors are fully faithful. Moreover, all functors in the displayed diagrams are symmetric monoidal. 
\end{prop}

Local duality allows us to replace the torsion categories in this result by their symmetric monoidally equivalent complete counterparts. We are now ready for the proof of our main comparison theorem. 

\begin{thm}\label{thm:ultracomparison}
For any non-principal ultrafilter $\cF$ on $\cP$, there is a commutative diagram of symmetric monoidal functors:
\[
\xymatrix{\Prod{\cF}^{\Pic}\Spnp \ar[r] \ar[d] & \Prod{\cF}^{\Pic}\cSpnp \ar[d] \\
\Prod{\cF}^{\Pic}\Frnp \ar[r] & \Prod{\cF}^{\Pic}\cFrnp.}
\]
\end{thm}
\begin{proof}
Consider the diagram of symmetric monoidal functors
\[
\xymatrix{\Prod{\cF}^{\Pic}\cSpnp \ar[rrr] \ar[ddd] & & & \Tot\Prod{\cF}^{\flat}\tMod_{\E^{\otimes \bullet + 1}} \ar[ddd] \\
& \Prod{\cF}^{\Pic}\Spnp \ar[r] \ar[d] \ar[lu] & \Tot\Prod{\cF}^{\flat}\Mod_{\E^{\otimes \bullet + 1}} \ar[d] \ar[ru] \\
& \Prod{\cF}^{\Pic}\Frnp \ar[r] \ar[dl] & \Tot\Prod{\cF}^{\flat}\Mod_{(\E^{\otimes \bullet + 1})_{\star}} \ar[dr] \\ 
\Prod{\cF}^{\Pic}\cFrnp \ar[rrr] & & & \Tot\Prod{\cF}^{\flat}\tMod_{(\E^{\otimes \bullet + 1})_{\star}}.}
\]
The top and bottom commutative squares have been constructed in \Cref{prop:torscomparisonsquare}, while the commutativity of the right square has been established in the proof of \Cref{prop:torsionformality}: Indeed, after post-composition with the inclusion functors we obtain a diagram of cosimplicial $\infty$-categories
\[
\xymatrix{\Prod{\cF}^{\flat}\Mod_{\E^{\otimes \bullet + 1}} \ar[r] \ar[d] & \Prod{\cF}^{\flat}\tMod_{\E^{\otimes \bullet + 1}} \ar[d] \ar[r] & \Prod{\cF}^{\flat}\Mod_{\E^{\otimes \bullet + 1}} \ar[d] \\
\Prod{\cF}^{\flat}\Mod_{(\E^{\otimes \bullet + 1})_{\star}} \ar[r] & \Prod{\cF}^{\flat}\tMod_{(\E^{\otimes \bullet + 1})_{\star}} \ar[r] & \Prod{\cF}^{\flat}\Mod_{(\E^{\otimes \bullet + 1})_{\star}}.}
\]
The right horizontal functors are fully faithful and the right square commutes by \Cref{prop:torsionformality}. Since the outer rectangle commutes when restricted to  $\Prod{\cF}\Cell_k(\Mod_{\E^{\otimes s + 1}}^{\omega})$ for any $k,s \ge 0$ by construction of the functors, it has to commute as well. Therefore, the left square commutes and passing to totalizations yields the desired commutativity of the right trapezoid in the larger diagram above.

The inner central and outer squares commute by the proofs of \cite[Theorem 5.39]{ultra1} and of \Cref{thm:knmain}, respectively. Since the bottom horizontal functor is fully faithful, it follows that the left trapezoid commutes as well. 
\end{proof}

\begin{rem}
One may wonder whether the canonical inclusions $\cSpnp \to \Spnp$ and $\cFrnp \to \Frnp$ assemble into a commutative diagram of Pic-generated protoproducts as well. This appears to be related to the following question: Given $p\in\cP$ sufficiently large and $n\ge 0$, does there exists a constant $N = N(n,p)$ such that for every $P \in \Pic(\cSpnp)$ and for a cofinal set of indices $I$, the spectra $M_p^I(n) \otimes P \in \Spnp^{\omega}$ are contained in $\Ind\Cell_N(\Spnp^{\omega})$?
\end{rem}

\bibliographystyle{amsalpha}
\bibliography{bibliography}

\end{document}